\documentclass[a4paper,12pt]{article}

\usepackage[utf8]{inputenc}
\usepackage[T1]{fontenc}
\usepackage[english]{babel}

\usepackage{amsmath,amsfonts,amssymb,amsthm}
\usepackage{graphicx,tabularx}
\usepackage[all]{xy}

\usepackage[hang]{caption}
\theoremstyle{definition}
\newtheorem{defn}{Definition}

\theoremstyle{plain}
\newtheorem{thm}{Theorem}

\newtheorem{lem}{Lemma}

\newcommand{\N}{\mathbb{N}}
\newcommand{\R}{\mathbb{R}}
\newcommand{\C}{\mathbb{C}}
\newcommand{\CC}{\widehat{\mathbb{C}}}
\renewcommand{\SS}{\mathbb{S}^{2}}
\newcommand{\D}{\mathbb{D}}

\newcommand{\JJ}{\mathcal{J}}
\newcommand{\JJcrit}{\mathcal{J}_{\mathrm{crit}}}
\newcommand{\HH}{\mathcal{H}}

\DeclareMathOperator{\modulus}{\mathrm{mod}}

\newcommand{\longlong}[1]{\overset{#1}{\relbar\joinrel\relbar\joinrel\longrightarrow}}

\title{A family of rational maps \\ with buried Julia components}
\author{Sébastien Godillon}
\date{November 24, 2013}

\begin{document}

%*******************************************************************************************

\maketitle

\begin{abstract}
It is known that the disconnected Julia set of any polynomial map does not contain buried Julia components. But such Julia components may arise for rational maps. The first example is due to Curtis T. McMullen who provided a family of rational maps for which the Julia sets are Cantor of Jordan curves. However all known examples of buried Julia components, up to now, are points or Jordan curves and comes from rational maps of degree at least 5.

This paper introduce a family of hyperbolic rational maps with disconnected Julia set whose exchanging dynamics of postcritically separating Julia components is encoded by a weighted dynamical tree. Each of these Julia sets presents buried Julia components of several types: points, Jordan curves, but also Julia components which are neither points nor Jordan curves. Moreover this family contains some rational maps of degree 3 with explicit formula that answers a question McMullen raised.
\end{abstract}

\newpage

\tableofcontents

\newpage

%*******************************************************************************************

\section{Introduction}\label{SecIntroduction}

For any rational map $f$ of degree $d\geqslant 2$ on the Riemann sphere $\CC$, we denote by $J(f)$ its Julia set, namely the closure of the set of repelling periodic points. We recall that $J(f)$ is a fully invariant non-empty perfect compact set which either is connected or has uncountably many connected components (see \cite{BeardonBook}, \cite{CarlesonGamelinBook}, \cite{MilnorBook}). This paper focuses on the disconnected case. Every connected component of $J(f)$ is called a Julia component and every connected component of the Fatou set $\CC-J(f)$ is called a Fatou domain.

A Julia component is said to be \textbf{buried} if it has no intersection with the boundary of any Fatou domain. In particular buried Julia components can not occur in the polynomial case (since the Julia set coincides with the boundary of the unbounded Fatou domain). The same holds if the Julia set is a Cantor set, or more generally if the complementary of every Julia component is connected (since the Fatou set is then connected). That suggests much more sophisticated topological structures for Julia sets with some buried Julia components than those encountered in the polynomial case.

The first example of rational maps with buried Julia components is due to Curtis T. McMullen. Consider the family of rational maps given by
$$g_{c,\lambda}:z\mapsto z^{d_{\infty}}+c+\frac{\lambda}{z^{d_{0}}}\quad\text{where}\ d_{\infty},d_{0}\geqslant 1\ \text{and}\ c,\lambda\in\C.$$
The special case $c=0$ has been studied in \cite{AutomorphismsRationalMaps} (see also \cite{SingularPerturbations}), where it is proved that if the following condition is satisfied
\begin{equation}\label{H0}
	\frac{1}{d_{\infty}}+\frac{1}{d_{0}}<1 \tag{H0}
\end{equation}
and if $|\lambda|>0$ is small enough then $J(g_{0,\lambda})$ is a Cantor of Jordan curves, namely homeomorphic to the product of a Cantor set with a Jordan curve (see Figure \ref{FigExamples}.a). Recall that any Cantor set is homeomorphic to the no middle third set on a line segment which contains uncountably many points which are not endpoints of any removing open segment. Each of these points corresponds to a buried Jordan curve in $J(g_{0,\lambda})$.

In \cite{RationalMapsDisconnectedJuliaSet}, the authors have provided another example by slightly modifying the map $g_{-1,\lambda}$ for $d_{\infty}=2$ and $d_{0}=3$ (that satisfies assumption (\ref{H0})) in a clever way:
$$\widetilde{g_{-1,\lambda}}:z\mapsto\frac{1}{z}\circ(z^{2}-1)\circ\frac{1}{z}+\frac{\lambda}{z^{3}}=\frac{z^{2}}{1-z^{2}}+\frac{\lambda}{z^{3}}\quad\text{where}\ \lambda\in\C.$$
If $|\lambda|>0$ is small enough then $J(\widetilde{g_{-1,\lambda}})$ has the same topological structure than $J(g_{0,\lambda})$ except that one fixed Julia component (which contains the boundary of the unbounded Fatou domain and hence is not buried) is quasiconformally homeomorphic to the Julia set of $z\mapsto z^{2}-1$. The uncountably many Julia components which are not eventually mapped under iteration onto this fixed Julia component are buried Jordan curves in $J(\widetilde{g_{-1,\lambda}})$ (see Figure \ref{FigExamples}.b).

Examples of buried Jordan components which are not Jordan curves have appeared in some works. For instance in \cite{GeneralizedMcMullenDomain} (see also \cite{SingularPerturbationsMcMullenDomain} and \cite{SingularPerturbationsQuadraticMultiplePoles}), the authors have studied the family $g_{c,\lambda}$  for $d_{\infty}=d_{0}\geqslant 3$ (that satisfies assumption (\ref{H0})) and for a fixed parameter $c$ chosen so that for the polynomial $z\mapsto z^{d_{\infty}}+c$ the critical point 0 lies in a cycle of period at least 2. In that case, if $|\lambda|>0$ is small enough then $J(g_{c,\lambda})$ still has uncountably many Jordan curves as buried Jordan components but also uncountably many points. The remaining Julia components are eventually mapped under iteration onto a fixed Julia component (which coincides with the boundary of the unbounded Fatou domain and hence is not buried) quasiconformally homeomorphic to the Julia set of $z\mapsto z^{d_{\infty}}+c$. Each of these not buried Julia components has infinitely many ``decorations'' and every buried point component is accumulated by a nested sequence of such decorations (see Figure \ref{FigExamples}.c).

\begin{figure}[!htb]
	\begin{center}
	\includegraphics[width=\textwidth]{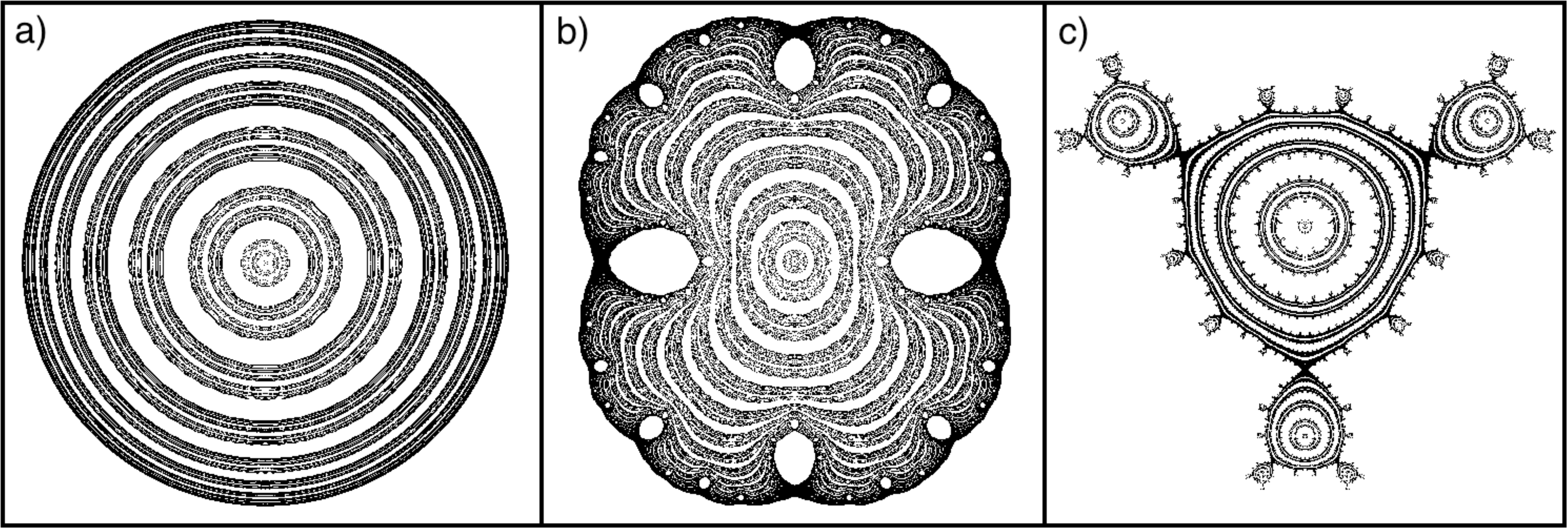}
	\caption{\textbf{a)} $J(g_{0,\lambda})$ for $d_{\infty}=2$, $d_{0}=3$ and $\lambda\approx10^{-9}$. \\ \textbf{b)} $J(\widetilde{g_{-1,\lambda}})$ for $d_{\infty}=2$, $d_{0}=3$ and $\lambda\approx10^{-9}$. \\ \textbf{c)} $J(g_{c,\lambda})$ for $d_{\infty}=d_{0}=3$, $c=-i$ and $\lambda\approx10^{-9}$.} 
	\label{FigExamples}
	\end{center}
\end{figure}

All the previous examples are rational maps of degree $d_{\infty}+d_{0}$ at least 5 according to assumption (\ref{H0}). The existence question of buried Julia components for rational maps of degree less than 5 has been raised in \cite{AutomorphismsRationalMaps}. In the last decade, a number of papers have appeared that deal with subfamilies of $g_{c,\lambda}$ or some slightly perturbations of it. Some of them present sophisticated Julia sets with buried Julia components, however the degree of these examples is always at least equal to 5. Furthermore the buried Julia components of these examples are points or Jordan curves.

The aim of this paper is to answer the question Curtis T. McMullen has raised by providing a family of rational maps of degree $3$ which does not come from the family $g_{c,\lambda}$ and whose Julia set presents buried Julia components of several types: points, Jordan curves but also Julia components which are neither points nor Jordan curves. One of our main result here is the following.

\begin{thm}\label{ThmPersianCarpet}
Consider the family of cubic rational maps given by
$$f_{\lambda}:z\mapsto\dfrac{(1-\lambda)\Big[(1-4\lambda+6\lambda^{2}-\lambda^{3})z-2\lambda^{3}\Big]}{(z-1)^{2}\Big[(1-\lambda-\lambda^{2})z-2\lambda^{2}(1-\lambda)\Big]}\quad\text{where}\ \lambda\in\C.$$
If $|\lambda|>0$ is small enough then $J(f_{\lambda})$ contains buried Julia components of several types: \\
\begin{tabularx}{\textwidth}{rX}
	\emph{(point type)} & uncountably many points; \\
	\emph{(circle type)} & uncountably many Jordan curves; \\
	\emph{(complex type)} & countably many preimages of a fixed Julia component which is quasiconformally homeomorphic to the connected Julia set of $f_{0}:z\mapsto\frac{1}{(z-1)^{2}}$.
\end{tabularx}
\end{thm}

An example of such Julia set is depicted in Figure \ref{FigPersianCarpet}. $J(f_{\lambda})$ is called a ``Persian carpet'' because of similarities with sophistications from carpet-weaving art: the Julia set of $f_{0}:z\mapsto\frac{1}{(z-1)^{2}}$ appears as a watermark in the central motif of the carpet whose surface is covered by an elaborate pattern of Cantor of Jordan curves, and there are some small Julia components everywhere that looks like dust. These small dusts contain nested sequences of finite coverings of the Persian carpet which accumulate buried point components.

\begin{figure}[!htb]
	\begin{center}
	\includegraphics[width=\textwidth]{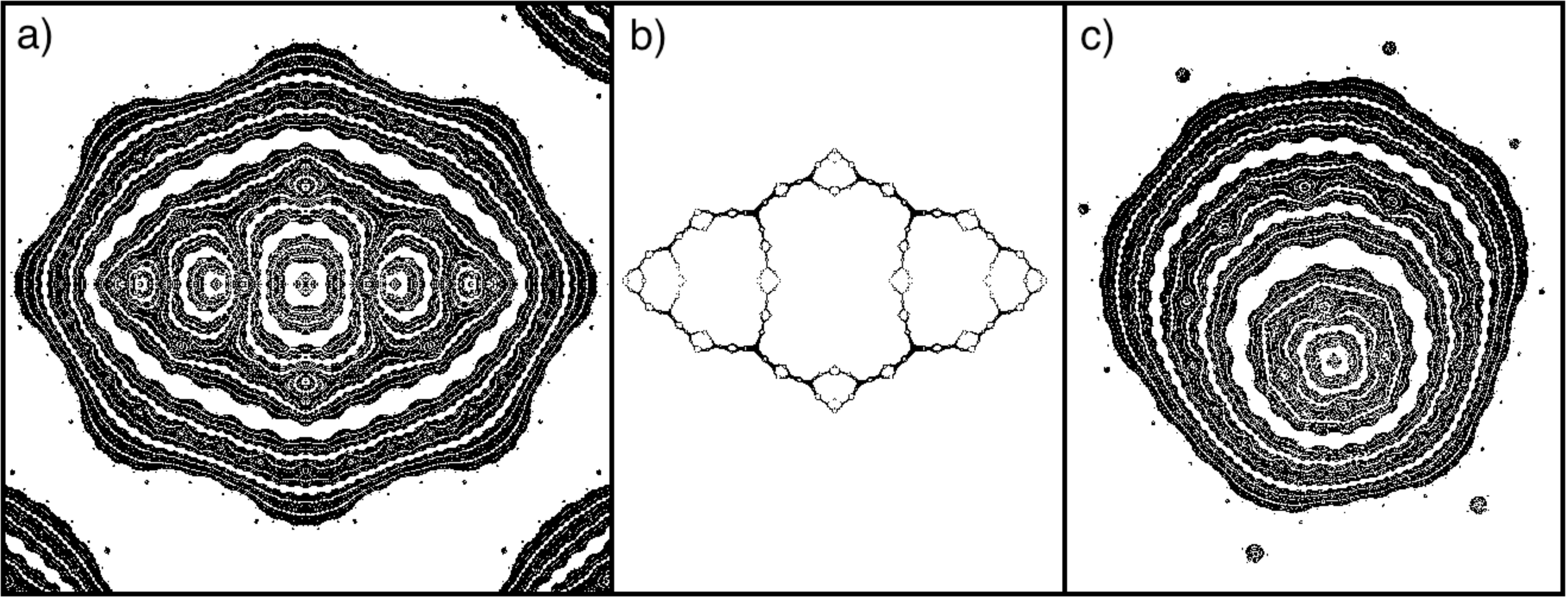}
\caption{\textbf{a)} A Persian carpet: $J(f_{\lambda})$ for $\lambda\approx10^{-3}$. \\ \textbf{b)} $J(f_{0})$ which appears as a buried Julia component in $J(f_{\lambda})$. \\ \textbf{c)} A magnification about a dust of the Persian carpet.}
	\label{FigPersianCarpet}
	\end{center}
\end{figure}

The Persian carpet example is maximal among rational maps with buried Julia components in the sense that buried Julia components can not occur for rational maps of degree less than $3$. Indeed, by a theorem in \cite{RationalMaps2CriticalPoints}, the Julia set of any quadratic rational maps is either connected or a Cantor set.

Furthermore, the Persian carpet example is maximal among geometrically finite rational maps (namely rational maps such that every critical point in the Julia set is preperiodic, in our case $f_{\lambda}$ is hyperbolic, namely it has no critical point in $J(f_{\lambda})$ for $|\lambda|>0$ small enough) in the sense that every Julia component (not necessarily buried) of such a map is one of the three types described in Theorem \ref{ThmPersianCarpet}. That follows from two results. Firstly, by a theorem in \cite{AutomorphismsRationalMaps}, every periodic Julia component of a rational map is either a point or quasiconformally homeomorphic to the connected Julia set of a rational map. Secondly, it has been proved in \cite{RationalMapsDisconnectedJuliaSet} that every Julia component of a geometrically finite rational map which is not eventually mapped under iteration onto a periodic Julia component is either a point or a Jordan curve.

The underlying idea in the construction of the Persian carpet example is that the sophisticated configuration on $\CC$ of Julia components which are not points may be encoded by a tree. Tree structures have appeared in various works on holomorphic dynamics (for instance Hubbard trees in \cite{DouadyHubbard1} to classify postcritically finite polynomial maps). The tree considered here is not embedded in $\CC$. It is seen as an abstract object which is very similar to, and actually inspired by, the trees introduced by Mitsuhiro Shishikura in \cite{TreesHermanRings} which describe the configurations of Herman rings for rational maps.

However, the purpose of this paper is only to introduce a family of rational maps coming from a particular tree which answers the question Curtis T. McMullen has raised. But not to discuss about the general existence question of rational maps whose configuration of Julia components is encoded by any given tree (that will be the purpose of future works) even if a general construction may be suggested (especially statements and discussions in Section \ref{SecEncoding}).

\textit{Organization of the paper.} Section \ref{SecEncoding} deals with exchanging dynamics of postcritically separating Julia components by weighted dynamical tree.

In Section \ref{SubSecMcMullen}, we specify the idea mentioned above by showing that, under assumption (\ref{H0}), the exchanging dynamics of Julia components for the family $g_{0,\lambda}$ is encoded by a certain weighted dynamical tree $(\HH_{Q},w)$ (see Theorem \ref{ThmMcMullen}).

The purpose of Section \ref{SubSecPersianCarpets}, is then to do the converse: starting from a particular dynamical tree $\HH_{P}$ more sophisticated than $\HH_{Q}$ and a weight function $w$ on its edges, Theorem \ref{ThmMain1} states the existence of rational maps with disconnected Julia set whose exchanging dynamics of postcritically separating Julia components is encoded by $(\HH_{P},w)$ if (and, actually, only if) two conditions (\ref{H1}) and (\ref{H2}) hold. Theorem \ref{ThmMain2} shows that the Julia sets of these rational maps own buried Julia components of every expected type.

The main part of the proofs of Theorem \ref{ThmMain1} and Theorem \ref{ThmMain2}, that is the construction by quasiconformal surgery of the required rational maps, is detailed in Section \ref{SecConstruction}.

In Section \ref{SecProperties}, some properties of the rational maps constructed in the previous section are shown. The properties about exchanging dynamics (Section \ref{SubSecExchangingDynamics}) conclude the proof of Theorem \ref{ThmMain1} while the properties about topology of some Julia components (Section \ref{SubSecTopology}) give the proof of Theorem \ref{ThmMain2}.

Section \ref{SecFormula} deals with a particular choice of the weight function $w$ for which the two assumptions (\ref{H1}) and (\ref{H2}) are satisfied and such that the rational maps in Theorem \ref{ThmMain1} and Theorem \ref{ThmMain2} have degree 3. In this case, an explicit formula is provided that concludes the proof of Theorem \ref{ThmPersianCarpet}.

Finally, some technical results used in the construction of Section \ref{SecConstruction} are collected in Section \ref{SecAppendix} with proofs or references.

\textit{Acknowledgment.} The author would like very much to thank Professor Tan Lei, the advisor of his thesis, together with Professor Cui Guizhen for all their useful comments and fruitful discussions on this work. Finally, the author thanks the referees for several helpful suggestions.

%*******************************************************************************************

\section{Encoding by weighted dynamical trees}\label{SecEncoding}

For any rational map $f:\CC\rightarrow\CC$, we denote by $\JJ(f)$ the set of Julia components and we recall that $f$ induces a topological dynamical system on $\JJ(f)$ endowed with the usual distance between continua on $\CC$ equipped with the spherical metric (notice that $\JJ(f)$ is closed for this distance, that is not true in general for the Hausdorff distance). This topological dynamical system is called the \textbf{exchanging dynamics of Julia components}.

We recall that the critical points of f are the points where f is not locally injective, and the postcritical points of f are the points of the form $f^{n}(c)$ for some $n\geqslant 1$ and for some critical point $c$. A Julia component $J\in\JJ(f)$ is said to be postcritically separating if $J$ separates the postcritical set of $f$, or equivalently if $\CC-J$ has at least two connected components containing at least one postcritical point of $f$ each. We denote by $\JJcrit(f)$ the subset of postcritically separating Julia components in $\JJ(f)$. Remark that $\JJcrit(f)$ is forward invariant, and thus $f$ induces a topological dynamical system on $\JJcrit(f)$.

\subsection{McMullen's example}\label{SubSecMcMullen}

Consider the cubic polynomial $Q:z\mapsto 3z^{2}(\frac{3}{2}-z)$. It has two simple critical points: 0 which is fixed, and 1 which is mapped on 0 after two iterations.
$$\xymatrix{0 \ar@(ul,dl)[]_{2:1} & 1 \ar[r]^{2:1} & \frac{3}{2} \ar@/_{2pc}/[ll]_{1:1}}$$

Let $\HH_{Q}$ be its Hubbard tree, namely the smallest closed connected infinite union of internal rays which contains the postcritical set $\{0,\frac{3}{2}\}$ (see \cite{DouadyHubbard1}). In fact, $\HH_{Q}$ is the straight real segment $[0,\frac{3}{2}]$ or more precisely the union of two edges $[0,1]\cup[1,\frac{3}{2}]$ while the vertices are $0$, $1$ and $\frac{3}{2}$. Both edges of $\HH_{Q}$ are homeomorphically mapped by $Q$ onto the whole tree (see Figure \ref{FigMcMullenTree}.b).

\begin{figure}[!htb]
	\begin{center}
	\includegraphics[width=\textwidth]{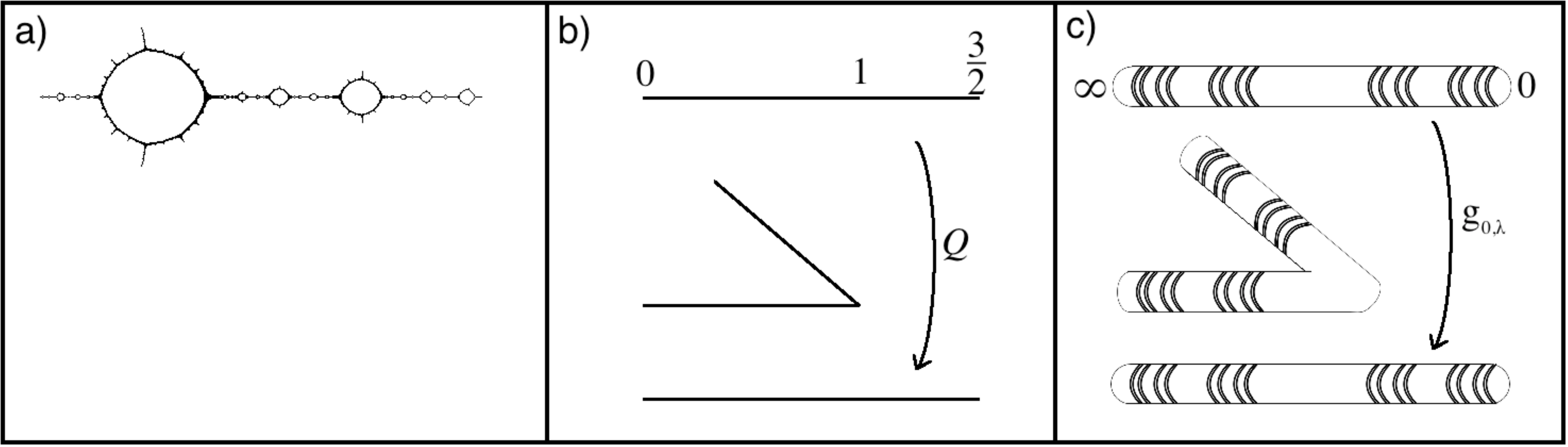}
	\caption{\textbf{a)} The Julia set of the polynomial $Q$. \\ \textbf{b)} The action of $Q$ on the Hubbard tree $\HH_{Q}$. \\ \textbf{c)} The action of $g_{0,\lambda}$ on the set of Julia components $\JJ(g_{0,\lambda})$.}
	\label{FigMcMullenTree}
	\end{center}
\end{figure}

Denote by $\JJ(\HH_{Q})$ the intersection set between the Hubbard tree $\HH_{Q}$ and the Julia set $J(Q)$. Notice that $\JJ(\HH_{Q})$ is disconnected (actually a Cantor set) and $Q$ induced a dynamical system on it since the Hubbard tree $\HH_{Q}$ and the Julia set $J(Q)$ are both invariant.

Finally, let $w$ be a weight function on the set of edges of $\HH_{Q}$, say $w([0,1])=d_{\infty}$ and $w([1,\frac{3}{2}])=d_{0}$ where $d_{\infty}$, and $d_{0}$ are positive integer.

The result about the family $g_{0,\lambda}$ discussed in the introduction (see Section \ref{SecIntroduction}) may be reformulated as follows.

\begin{thm}\label{ThmMcMullen}
If the weighted dynamical tree $(\HH_{Q},w)$ satisfies the following condition
\begin{equation}
	\frac{1}{d_{\infty}}+\frac{1}{d_{0}}<1
	\tag{H0}
\end{equation}
then for every $|\lambda|>0$ small enough, the exchanging dynamics of Julia component of $g_{0,\lambda}$ is encoded by $(\HH_{Q},w)$ in the following sense:
\begin{description}
	\item[(i)] every critical orbit accumulates the super-attracting fixed point $\infty$;
	\item[(ii)] there exists a homeomorphism $h:\JJ(g_{0,\lambda})\rightarrow\JJ(\HH_{Q})$ such that the following diagram commutes;
$$\xymatrix{
	\JJ(g_{0,\lambda}) \ar[rr]^{\textstyle g_{0,\lambda}} \ar[d]_{\textstyle h} && \JJ(g_{0,\lambda}) \ar[d]^{\textstyle h} \\
	\JJ(\HH_{Q}) \ar[rr]_{\textstyle Q} && \JJ(\HH_{Q})}$$
	\item[(iii)] for every Julia component $J\in\JJ(g_{0,\lambda})$, the restriction map $g_{0,\lambda}|_{J}$ has degree $w(e)$ where $e$ is the edge of $\HH_{Q}$ which contains $h(J)$.
\end{description}
\end{thm}

Notice that $\JJ(g_{0,\lambda})=\JJcrit(g_{0,\lambda})$ for $|\lambda|>0$ small enough since every Julia component is a Jordan curve which separates the fixed critical point $\infty$ from some critical values close to $0$.

\begin{proof}
We only sketch the proof since the main part is done in \cite{AutomorphismsRationalMaps}. Indeed it is shown that there exists a large annulus $A$ centered at 0 and containing $J(g_{0,\lambda})$ whose preimage consists of two disjoint annuli $A_{\infty},A_{0}$ both nested in $A$ and such that the restriction maps $g_{0,\lambda}|A_{\infty}:A_{\infty}\rightarrow A$ and $g_{0,\lambda}|A_{0}:A_{0}\rightarrow A$ are coverings of degree $d_{\infty}$ and $d_{0}$, respectively. Using combinatorial reasoning from holomorphic dynamics, it is a classical exercise to prove that the set of connected components of $J(g_{0,\lambda})=\cap_{n\geqslant 0}g_{0,\lambda}^{-n}(A)$ is homeomorphic to the space of all sequences of two digits $\Sigma_{2}=\{0,1\}^{\N}$ (equipped with the product topology making it a Cantor set) and the exchanging dynamics is topologically conjugated to a 2-to-1 shift map $\sigma:\Sigma_{2}\rightarrow\Sigma_{2}$ defined by $\sigma(s_{0},s_{1},s_{2},\dots)=(s_{1},s_{2},s_{3},\dots)$. The same holds for the dynamical system induced by $Q$ on $\JJ(\HH_{Q})$ since for $\varepsilon>0$ small enough the real segment $I=[\varepsilon,\frac{3}{2}-\varepsilon]$ contains $\JJ(\HH_{Q})$ and its preimage consists of two disjoint real segment both included in $I$ (one in each of the two edges of $\HH_{Q}$).
\end{proof}

Heuristically speaking, we may topologically think the Riemann sphere $\CC$ as a smooth neighborhood's boundary of the tree $\HH_{Q}$ embedded in the space $\R^{3}$. The two points on this topological sphere which correspond to $\infty$ and $0$ should be closed to the corresponding vertices of $\HH_{Q}$ which are $0$ and $\frac{3}{2}$, respectively. If the neighborhood becomes smaller and smaller, every Jordan curves in $J(g_{0,\lambda})$ is shrunk to a point in $J(\HH_{Q})$ (see Figure \ref{FigMcMullenTree}.c).

%*******************************************************************************************

\subsection{Persian carpet example}\label{SubSecPersianCarpets}

Consider a quadratic polynomial of the form $P:z\mapsto z^{2}+c$ where the parameter $c\in\C$ is chosen in order that the critical point 0 is periodic of period 4. There are exactly six choices of such a parameter. Let us fix $c$ to be that one with the largest imaginary part, that is $c\approx-0.157+1.032i$. The postcritical points are denoted by $c_{k}=P^{k}(0)$ for every $k\in\{0,1,2,3\}$.
$$\xymatrix{
	&& c_{1} \ar@/_{1pc}/[lld]_{1:1} & \\
	c_{2} \ar@/_{2pc}/[rrrdd]_{1:1} & \alpha \ar@(dr,ur)[]_{1:1} && \\
	&& c_{0} \ar@/_{1pc}/[uu]_{2:1} & \\
	&&& c_{3} \ar@/_{1pc}/[lu]_{1:1}
}$$

Let $\HH_{P}$ be the Hubbard tree of $P$ (see Figure \ref{FigPersianCarpetTree}.b). As one-dimensional simplicial complex, $\HH_{P}$ may be described by a set of five vertices $\{c_{0},c_{1},c_{2},c_{3},\alpha\}$ where $\alpha$ is a fixed point of $P$ and the following four edges:
$$e_{0}=[\alpha,c_{0}]_{\HH_{P}};\ e_{1}=[\alpha,c_{1}]_{\HH_{P}};\ e_{2}=[\alpha,c_{2}]_{\HH_{P}};\ e_{3=}[c_{0},c_{3}]_{\HH_{P}}.$$
$P$ homeomorphically acts on the edges as follows.
$$\left\{\begin{array}{l}
	P(e_{0})=e_{1} \\
	P(e_{1})=e_{2} \\
	P(e_{2})=e_{0}\cup e_{3} \\
	P(e_{3})=e_{0}\cup e_{1}
\end{array}\right.$$
Denote by $\JJ(\HH_{P})$ the intersection set between the Hubbard tree $\HH_{P}$ and the Julia set $J(P)$. Notice that $\JJ(\HH_{P})$ is disconnected (actually a Cantor set) and $P$ induced a dynamical system on it. Moreover the fixed branching point $\alpha$ belongs to $\JJ(\HH_{P})$ but not to the boundary of any connected component of $\HH_{P}-\JJ(\HH_{P})$. Finally, let $w$ be a weight function on the set of edges of $\HH_{P}$, say $w(e_{k})=d_{k}$ where $d_{k}$ is a positive integer for every $k\in\{0,1,2,3\}$.

\begin{figure}[!htb]
	\begin{center}
	\includegraphics[width=\textwidth]{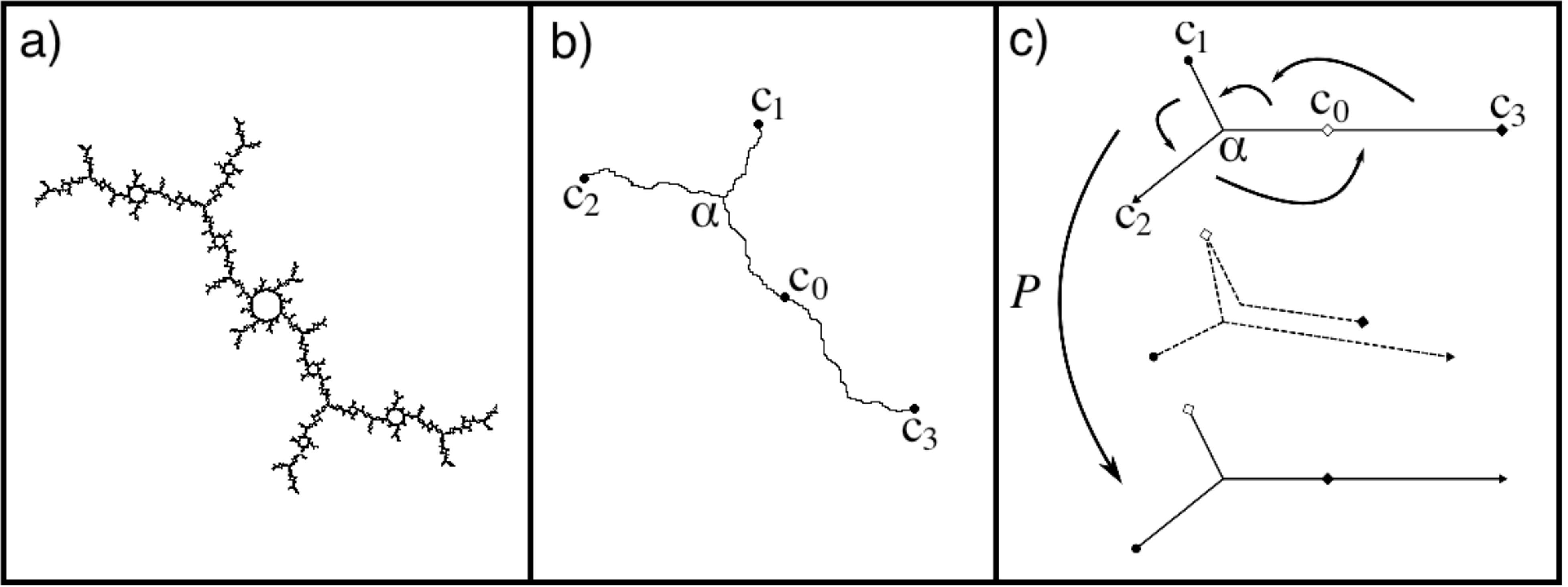}
	\caption{\textbf{a)} The Julia set of the polynomial $P$. \\ \textbf{b)} The Hubbard tree $\HH_{P}$. \\ \textbf{c)} The action of $P$ on a straightened copy of $\HH_{P}$.}
	\label{FigPersianCarpetTree}
	\end{center}
\end{figure}

\begin{defn}\label{DefTransitionMatrix}
The transition matrix of the weighted dynamical tree $(\HH_{P},w)$ is the $4$-by-$4$ matrix $M=(m_{i,j})_{i,j\in\{0,1,2,3\}}$ whose entries are defined as follows.
$$\forall i,j\in\{0,1,2,3\},\ m_{i,j}=\left\{\begin{array}{cl}
	\dfrac{1}{w(e_{i})} & \text{if}\ e_{j}\subset P(e_{i}) \\
	0 & \text{otherwise}
\end{array}\right.$$
Since $M$ is a non-negative matrix, it follows from Perron-Frobenius Theorem that the eigenvalue with the largest modulus is real and non-negative. Let us call $\lambda(\HH_{P},w)$ this leading eigenvalue. The weighted dynamical tree $(\HH_{P},w)$ is said to be \textbf{unobstructed} if $\lambda(\HH_{P},w)<1$.
\end{defn}

Let us give some remarks about this definition.
\begin{enumerate}
	\item This definition is strongly related to obstructions which occur in Thurston characterization of postcritically finite rational maps and all the theory behind (see \cite{ThurstonProof})
	\item When $(\HH_{P},w)$ is unobstructed, Perron-Frobenius Theorem and continuity of the spectral radius ensure the existence of a vector $V\in\R^{4}$ with positive entries such that $MV<V$. This remark will be useful later.
	\item Actually the transition matrix of $(\HH_{P},w)$ is given by
$$M=\left(\begin{array}{cccc}
	0 & \frac{1}{d_{0}} & 0 & 0 \\
	0 & 0 & \frac{1}{d_{1}} & 0 \\
	\frac{1}{d_{2}} & 0 & 0 & \frac{1}{d_{2}} \\
	\frac{1}{d_{3}} & \frac{1}{d_{3}} & 0 & 0
\end{array}\right)$$
and an easy computation shows that $\lambda(\HH_{P},w)$ is the largest root of
$$X^{4}-\left(\frac{1}{d_{0}d_{1}d_{2}}+\frac{1}{d_{1}d_{2}d_{3}}\right)X-\frac{1}{d_{0}d_{1}d_{2}d_{3}}.$$
Notice that if $\lambda(\HH_{P},w)\geqslant 1$ then $\lambda(\HH_{P},w)\leqslant\frac{1}{d_{0}d_{1}d_{2}}+\frac{1}{d_{1}d_{2}d_{3}}+\frac{1}{d_{0}d_{1}d_{2}d_{3}}$, thus $(\HH_{P},w)$ is unobstructed as soon as at least three of weights $d_{0}$, $d_{1}$, $d_{2}$, and $d_{3}$ are $\geqslant 2$. Conversely, if $(\HH_{P},w)$ is unobstructed then one can show by exhaustion that at least two of weights $d_{0}$, $d_{1}$, $d_{2}$, and $d_{3}$ are $\geqslant 2$.
	\item For the McMullen's example, the transition matrix of $(\HH_{Q},w)$ may be defined as well and we get
$$M=\left(\begin{array}{cc}
	\frac{1}{d_{\infty}} & \frac{1}{d_{\infty}} \\
	\frac{1}{d_{0}} & \frac{1}{d_{0}}
\end{array}\right).$$
An easy computation gives that $\lambda(\HH_{Q},w)=\frac{1}{d_{\infty}}+\frac{1}{d_{0}}$. Consequently the weighted dynamical tree $(\HH_{Q},w)$ is unobstructed if and only if the assumption (\ref{H0}) holds.
\end{enumerate}

The following result is analogous to Theorem \ref{ThmMcMullen}.

\begin{thm}\label{ThmMain1}
If the weighted dynamical tree $(\HH_{P},w)$ satisfies the two following conditions
\begin{equation}\label{H1}
	\widehat{d}=\frac{1}{2}(d_{0}+d_{1}+d_{2}-1)\ \text{is an integer}\geqslant 2\ \text{and}\ \max\{d_{0},d_{1},d_{2}\}\leqslant\widehat{d}
	\tag{H1}
\end{equation}
\begin{equation}\label{H2}
	(\HH_{P},w)\ \text{is unobstructed}
	\tag{H2}
\end{equation}
then there exists a rational map $f$ of degree $\widehat{d}+d_{3}$ such that the exchanging dynamics of postcritically separating Julia components of $f$ is encoded by $(\HH_{P},w)$ in the following sense:
\begin{description}
	\item[(i)] every critical orbit accumulates a super-attracting cycle $\{z_{0},z_{1},z_{2},z_{3}\}$ of period $4$;
	\item[(ii)] there exists a homeomorphism $h:\JJcrit(f)\rightarrow\JJ(\HH_{P})$ such that the following diagram commutes;
$$\xymatrix{
	\JJcrit(f) \ar[rr]^{\textstyle f} \ar[d]_{\textstyle h} && \JJcrit(f) \ar[d]^{\textstyle h} \\
	\JJ(\HH_{P}) \ar[rr]_{\textstyle P} && \JJ(\HH_{P})
}$$
	\item[(iii)] for every Julia component $J\in \JJcrit(f)$ such that $h(J)$ is not eventually mapped under iteration to the fixed branching point $\alpha$, the restriction map $f|_{J}$ has degree $w(e_{k})=d_{k}$ where $e_{k}$ is the edge of $\HH_{P}$ which contains $h(J)$.
\end{description}
\end{thm}

The same heuristic as for Theorem \ref{ThmMcMullen} still holds: we may topologically think the Riemann sphere $\CC$ as a smooth neighborhood's boundary of the tree $\HH_{P}$ embedded in the space $\R^{3}$. The action of $f$ on this topological sphere follows that one of the dynamical tree $\HH_{P}$. The points on this topological sphere which correspond to the points in the super-attracting periodic cycle $\{z_{0},z_{1},z_{2},z_{3}\}$ should be closed to the corresponding vertices $\{c_{0},c_{1},c_{2},c_{3}\}$ of $\HH_{P}$, and every Julia component in $\JJcrit(f)$ closely surrounds a corresponding point in $\JJ(\HH_{P})$.

\begin{thm}\label{ThmMain2}
Under assumptions (\ref{H1}) and (\ref{H2}) there exists a rational map $f$ satisfying Theorem \ref{ThmMain1} and such that $J(f)$ contains buried Julia components of several types: \\
\begin{tabularx}{\textwidth}{rX}
	\emph{(point type)} & uncountably many points; \\
	\emph{(circle type)} & uncountably many Jordan curves; \\
	\emph{(complex type)} & countably many preimages of a fixed Julia component lying over the fixed branching point $\alpha$, say $J_{\alpha}=h^{-1}(\alpha)\in\JJ(f)$, which is quasiconformally homeomorphic to the connected Julia set of a rational map $\widehat{f}$.
\end{tabularx} \\
Moreover $\widehat{f}$ has degree $\widehat{d}$ and has only one critical orbit which is a super-attracting cycle $\{\widehat{z_{0}},\widehat{z_{1}},\widehat{z_{2}}\}$ of period $3$ such that the local degree of $\widehat{f}$ at $\widehat{z_{k}}$ is $d_{k}$ for every $k\in\{0,1,2\}$.
\end{thm}

Let us give some comments about these results.
\begin{enumerate}
	\item The rational map $\widehat{f}$ corresponds to the dynamics of $f$ on the fixed Julia component $J_{\alpha}$ lying over the fixed branching point $\alpha$. More precisely, there is a quasiconformal map $\varphi$ from a neighborhood of $J(\widehat{f})$ onto a neighborhood of $J_{\alpha}$ such that $\varphi\circ\widehat{f}=f\circ\varphi$ (see the construction of $f$ in Section \ref{SecConstruction}).
 	\item The rational map $\widehat{f}$ may also be seen as encoded by a weighted dynamical tree. Consider the quadratic polynomial $R:z\mapsto z^{2}+\widehat{c}$ where $\widehat{c}\in\C$ is the parameter with the largest imaginary part such that the critical point 0 is periodic of period 3, that is $\widehat{c}\approx-0.123+0.745i$ ($J(R)$ is known as the Douady's rabbit). The Hubbard tree $\HH_{R}$ of $R$ is described by a set of four vertices $\{\widehat{c_{0}},\widehat{c_{1}},\widehat{c_{2}},\widehat{\alpha}\}$ where $\widehat{c_{k}}=R^{k}(0)$ and $\widehat{\alpha}$ is a fixed point of $R$, and three edges of the form $\widehat{e_{k}}=[\widehat{\alpha},\widehat{c_{k}}]_{\HH_{R}}$ for every $k\in\{0,1,2\}$. Consider the weight function $w$ defined by $w(\widehat{e_{k}})=d_{k}$ for every $k\in\{0,1,2\}$. Then the weighted dynamical tree $(\HH_{R},w)$ encodes the action of $\widehat{f}$ in the same setting as in Theorem \ref{ThmMcMullen} and Theorem \ref{ThmMain1}. Notice that the intersection set between $\HH_{R}$ and $J(R)$ is reduced to $\JJ(\HH_{R})=\{\widehat{\alpha}\}$, that corresponds to the unique Julia component in $\JJ(\widehat{f})=\JJcrit(\widehat{f})=\{J(\widehat{f})\}$. Finally, remark that the weighted dynamical tree $(\HH_{R},w)$ is unobstructed as soon as assumption (\ref{H1}) holds (actually $\lambda(\HH_{R},w)=\frac{1}{d_{0}d_{1}d_{2}}$).
	\item The rational map $\widehat{f}$ is unique up to conjugation by a Möbius map or equivalently it is unique as soon as its critical orbit $\{\widehat{z_{0}},\widehat{z_{1}},\widehat{z_{2}}\}$ is fixed in $\CC$ (see Lemma \ref{LemBranchingMap}). However, the rational map $f$ is not unique since the critical points which do not belong to the super-attracting periodic cycle $\{z_{0},z_{1},z_{2},z_{3}\}$ (but whose orbits accumulate it) may be perturbed in some neighborhoods without changing the exchanging dynamics and the topology of Julia components.
	\item The rational map $f$ is not postcritically finite since $J(f)$ is disconnected (but it is hyperbolic from point \textbf{(i)} in Theorem \ref{ThmMain1}). In particular Thurston characterization of postcritically finite rational maps (see \cite{ThurstonProof}) can not be used to prove the existence of $f$. However one could use the works of Tan Lei and Cui Guizhen about sub- hyperbolic semi-rational maps in \cite{ThurstonSubHyperbolic} but this paper presents a more explicit and more constructive method by quasiconformal surgery (see Section \ref{SecConstruction}).
	\item The assumption (\ref{H1}) is necessary. Indeed it is the smallest requirement such that there exists a topological model for $\widehat{f}$, that is a branched covering combinatorially equivalent to $\widehat{f}$ (see Lemma \ref{LemHurwitz} and proof of Lemma \ref{LemBranchingMap}).
	\item The assumption (\ref{H2}) is necessary. Otherwise we can find a Thurston obstruction, that is to say a multicurve $\Gamma$ whose transition matrix is equal to $M$ with leading eigenvalue $\lambda(\Gamma)=\lambda(\HH_{P},w)\geqslant 1$. According to a result of Curtis T. McMullen in \cite{McMullenBook} it follows that $\lambda(\Gamma)=1$ and at least one curve in $\Gamma$ is contained in an union of Fatou domains where $f$ is biholomorphically conjugated to a rotation. That is a contradiction since every critical orbit of $f$ accumulates a super-attracting periodic cycle.
\end{enumerate}

%*******************************************************************************************

\section{Construction}\label{SecConstruction}

The aim of this section is to construct by quasiconformal surgery (we refer readers to \cite{QuasiconformalSurgeryBook} for a comprehensive treatment on this powerful method) a rational map $f$ which sastifies Theorem \ref{ThmMain1} and Theorem \ref{ThmMain2}. The strategy is to start from a rational map $\widehat{f}$ whose Julia set corresponds to the branching point $\alpha$ in $\HH_{P}$ (see Theorem \ref{ThmMain2}) and then to modify this map in order to create a folding corresponding to the critical point $c_{0}$.

\subsection{The branching map $\widehat{f}$}

The first step of the construction is to prove the existence of the rational map $\widehat{f}$ which appears in Theorem \ref{ThmMain2}. This is done by Lemma \ref{LemBranchingMap} below.

\begin{lem}\label{LemBranchingMap}
If assumption (\ref{H1}) holds then there exists a rational map $\widehat{f}:\CC\rightarrow\CC$ of degree $\widehat{d}$ such that:
\begin{description}
	\item[(i)] $\widehat{f}$ has only one critical orbit which is a super-attracting cycle $\{\widehat{z_{0}},\widehat{z_{1}},\widehat{z_{2}}\}$ of period 3 such that the local degree of $\widehat{f}$ at $\widehat{z_{k}}$ is $d_{k}$ for every $k\in\{0,1,2\}$;
	\item[(ii)] $J(\widehat{f})$ is connected and the Fatou set $\CC-J(\widehat{f})$ has infinitely many connected components which are simply connected.
\end{description}
Moreover $\widehat{f}$ is unique up to conjugation by a Möbius map.
\end{lem}

There are many ways to prove the existence of $\widehat{f}$ (for instance by ``blowing up'' the edges of some triangle invariant by a Möbius map, see \cite{CombiningRationalMaps}). Here we give a simple proof provided a particular solution of the Hurwitz problem (see Section \ref{SecAppendix}).

\begin{proof}
Up to conjugation by a Möbius map, we may fix three distinct points $\widehat{z_{0}}$, $\widehat{z_{1}}$, and $\widehat{z_{2}}$ in $\CC$. Remark that if at least one of integers $d_{0}$, $d_{1}$, and $d_{2}$ is equal to $1$, says $d_{0}=1$, then assumption (\ref{H1}) leads to $d_{1}=d_{2}=\widehat{d}$ and the rational map $\widehat{f}=\varphi\circ(z\mapsto z^{\widehat{d}})\circ\widetilde{\varphi}^{-1}$ where $\varphi$ and $\widetilde{\varphi}$ are two Möbius maps such that
$$\begin{array}{cccc}
	& \widetilde{\varphi}(1)=\widehat{z_{0}}, & \widetilde{\varphi}(0)=\widehat{z_{1}}, & \widetilde{\varphi}(\infty)=\widehat{z_{2}}, \\
	\text{and} & \varphi(1)=\widehat{z_{1}}, & \varphi(0)=\widehat{z_{2}}, & \varphi(\infty)=\widehat{z_{2}},
\end{array}$$
satisfies \textbf{(i)}. Consequently we may assume that $d_{0}$, $d_{1}$, and $d_{2}$ are $\geqslant 2$.

If follows that we may apply Lemma \ref{LemHurwitz} since assumption (\ref{H1}) easily implies condition (\ref{H1'}) for the abstract branch data coming from $d=\widehat{d}$, and $d_{i,1}=d_{i-1}$ for every $i\in\{1,2,3\}$. We get a degree $\widehat{d}$ branched covering $H:\SS\rightarrow\SS$ and three distinct points $x_{1,1}$, $x_{2,1}$, and $x_{3,1}$ in $\SS$ such that the local degree of $H$ at $x_{i,1}$ is $d_{i-1}$ for every $i\in\{1,2,3\}$ and $H$ has no more critical points than $x_{1,1}$, $x_{2,1}$, and $x_{3,1}$. Let $\varphi:\SS\rightarrow\CC$ be any homeomorphism such that $\varphi(H(x_{i,1}))=\widehat{z_{i}}$ for every $i\in\{1,2,3\}$. Remark that the map $\varphi\circ H:\SS\rightarrow\CC$ induces a complex structure on $\SS$. In other words, the uniformization theorem gives a homeomorphism $\widetilde{\varphi}:\SS\rightarrow\CC$ such that the map $\widehat{f}=\varphi\circ H\circ\widetilde{\varphi}^{-1}$ is holomorphic on $\CC$ and thus a rational map of degree $\widehat{d}$. Moreover, up to postcomposition with a Möbius map, we may assume that $\widetilde{\varphi}(x_{i,1})=\widehat{z_{i-1}}$ for every $i\in\{1,2,3\}$ so that $\widehat{f}$ satisfies \textbf{(i)}.

Now remark that for every $k\in\{0,1,2\}$, the connected component containing $\widehat{z_{k}}$ of the super-attracting basin of $\widehat{f}$ is simply connected since it contains at most one critical point. Moreover, any other Fatou component is eventually mapped by homeomorphisms onto one of these simply connected components. It follows that $\widehat{f}$ satisfies \textbf{(ii)}.

Finally let $\widehat{g}$ be another rational map of degree $\widehat{d}$ which satisfies \textbf{(i)} and \textbf{(ii)} for the same super-attracting periodic cycle $\{\widehat{z_{0}},\widehat{z_{1}},\widehat{z_{2}}\}$. Then $z\mapsto\widehat{f}(z)-\widehat{g}(z)$ is a rational map of degree at most $2\widehat{d}$ for which 0 has at least $d_{0}+d_{1}+d_{2}=2\widehat{d}+1$ preimages counted with multiplicity (every $\widehat{z_{k}}$ is a preimage of 0 with multiplicity $d_{k}$). Consequently this map is identically equal to 0, that is $\widehat{f}=\widehat{g}$.
\end{proof}

Notice that the previous proof strongly uses the fact that the postcritical set contains only three points. Indeed if the postcritical set contains more than three points, there is still an uniformization map $\widetilde{\varphi}$ for $\SS$ equipped with the complex structure coming from $\varphi\circ H$, but that may not be possible to postcompose $\widetilde{\varphi}$ with a Möbius map so that $\widehat{f}$ satisfies \textbf{(i)}. In fact we would also need to check that the branched covering $H$ has no Thurston obstructions (see \cite{ThurstonProof}).

\subsection{Cutting along a system of equipotentials}

Starting with the map $\widehat{f}$ coming from Lemma \ref{LemBranchingMap}, we need to divide $\CC$ into several pieces on which the map $f$ (or more precisely a quasiregular map $F$) will be piecewisely defined. This partition comes from a certain system of equipotentials of $\widehat{f}$ defined in Lemma \ref{LemEquip} below.

For every $k\in\{0,1,2\}$, denote by $B(\widehat{z_{k}})$ the connected component containing $\widehat{z_{k}}$ of the super-attracting basin of $\widehat{f}$. Recall that each $B(\widehat{z_{k}})$ is a marked hyperbolic disk. More precisely, Böttcher's Theorem provides Riemann mappings $\phi_{k}:\D\rightarrow B(\widehat{z_{k}})$ (namely biholomorphic maps from the open unit disk $\D$ onto $B(\widehat{z_{k}}$) such that $\phi_{k}(0)=\widehat{z_{k}}$ and the following diagram commutes.
$$\xymatrix{
	B(\widehat{z_{0}}) \ar[d]_{\textstyle \widehat{f}} && \D \ar[ll]_{\textstyle \phi_{0}} \ar[d]^{\textstyle z\mapsto z^{d_{0}}} \\
	B(\widehat{z_{1}}) \ar[d]_{\textstyle \widehat{f}} && \D \ar[ll]_{\textstyle \phi_{1}} \ar[d]^{\textstyle z\mapsto z^{d_{1}}} \\
	B(\widehat{z_{2}}) \ar[d]_{\textstyle \widehat{f}} && \D \ar[ll]_{\textstyle \phi_{2}} \ar[d]^{\textstyle z\mapsto z^{d_{2}}} \\
	B(\widehat{z_{0}}) && \D \ar[ll]_{\textstyle \phi_{0}}
}$$

Recall that an equipotential $\beta$ in any $B(\widehat{z_{k}})$ is the image by $\phi_{k}$ of an euclidean circle in $\D$ centered at $0$. The radius of this circle is called the level of $\beta$ and is denoted by $L_{k}(\beta)\in]0,1[$, in order that
$\beta=\{z\in B(\widehat{z_{k}})\,/\,|\phi_{k}^{-1}(z)|=L_{k}(\beta)\}$.

Recall that any pair of disjoint continua $\beta,\beta'$ in $\CC$ uniquely defines an open annulus in $\CC$ denoted by $A(\beta,\beta')$. If $\beta,\beta'$ contain at least two points each, $A(\beta,\beta')$ is biholomorphic to a round annulus of the form $A_{r}=\{z\in\C\,/\,r<|z|<1\}$ where $r\in]0,1[$ only depends on $A(\beta,\beta')$. The modulus of $A(\beta,\beta')$ is defined to be $\modulus(A(\beta,\beta'))=\frac{1}{2\pi}\log(\frac{1}{r})$. In particular if $\beta,\beta'$ are two equipotentials in the same domain $B(\widehat{z_{k}})$ of levels $L_{k}(\beta)>L_{k}(\beta')$ then
$$\modulus(A(\beta,\beta'))=\frac{1}{2\pi}\log\left(\frac{L_{k}(\beta)}{L_{k}(\beta')}\right).$$

Finally for every $k\in\{0,1,2\}$, denote by $\alpha_{k}$ the compact connected subset of $J(\widehat{f})$ which corresponds to the boundary of $B(\widehat{z_{k}})$.

\begin{lem}\label{LemEquip}
If assumption (\ref{H2}) holds then there exist three equipotentials $\beta_{0}$ in $B(\widehat{z_{0}})$, $\beta_{1}$ in $B(\widehat{z_{1}})$, and $\beta_{2}$ in $B(\widehat{z_{2}})$, together with two equipotentials $\beta_{3}^{+}$ and $\beta_{3}^{-}$ in $B(\widehat{z_{0}})$ such that 
$$L_{0}(\beta_{0})>L_{0}(\beta_{3}^{+})>L_{0}(\beta_{3}^{-})$$
and the following linear system of inequalities holds.
\begin{equation}\label{EqSysIneq}
	\left\{ \begin{array}{ccl}
		\dfrac{1}{d_{0}}\modulus(A(\alpha_{1},\beta_{1})) & < & \modulus(A(\alpha_{0},\beta_{0})) \\
		\dfrac{1}{d_{1}}\modulus(A(\alpha_{2},\beta_{2})) & < & \modulus(A(\alpha_{1},\beta_{1})) \\
		\dfrac{1}{d_{2}}\modulus(A(\alpha_{0},\beta_{0})) + \dfrac{1}{d_{2}}\modulus(A(\beta_{3}^{+},\beta_{3}^{-})) & < & \modulus(A(\alpha_{2},\beta_{2})) \\
		\dfrac{1}{d_{3}}\modulus(A(\beta_{1},\beta_{0})) & < & \modulus(A(\beta_{3}^{+},\beta_{3}^{-})) \\
		&& \\
		\multicolumn{3}{c}{\text{and}\quad\modulus(A(\beta_{0},\beta_{3}^{+}))>1}\\
	\end{array}\right.
\end{equation}
\end{lem}

Recall that the modulus is a conformal invariant, or more precisely if there is a holomorphic covering of degree $d$ from an open annulus $A$ onto another one $A'$ then $\modulus(A)=\frac{1}{d}\modulus(A')$. Hence the first three inequalities in linear system (\ref{EqSysIneq}) implies that the preimages under $\widehat{f}$ of these equipotentials are arranged as shown in Figure \ref{FigEquipotentials}. The fourth inequality will allow to realize the preimage of the branching point $\alpha$ in $\HH_{P}$ (see Lemma \ref{LemPreimage}) while the last inequality ensures sufficient space to realize the folding corresponding to the critical point $c_{0}$ (see Lemma \ref{LemFolding}).

\begin{figure}[!htb]
	\begin{center}
	\includegraphics[width=\textwidth]{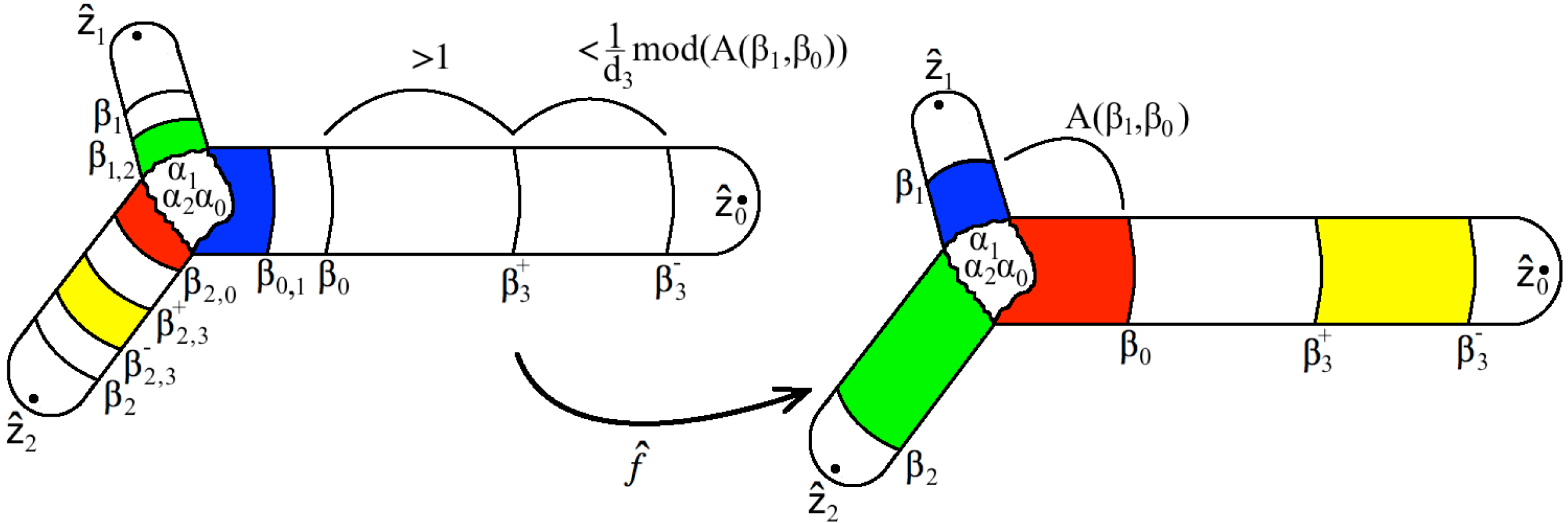}
	\caption{The pattern of the equipotentials (and their preimages) coming from Lemma \ref{LemEquip} displayed on the Riemann sphere which is topologically distorted to emphasize the three domains $B(\widehat{z_{0}})$, $B(\widehat{z_{1}})$, and $B(\widehat{z_{2}})$ (compare with Figure \ref{FigPersianCarpetTree}.c).}
	\label{FigEquipotentials}
	\end{center}
\end{figure}

The key point of the proof needs an inverse Grötzch's inequality due to Cui Guizhen and Tan Lei (see Section \ref{SecAppendix}).

\begin{proof}
Let $C>0$ be the constant coming from Lemma \ref{LemGrotzch} for the marked hyperbolic disks $B(\widehat{z_{0}}),B(\widehat{z_{1}})$. Thus, for every pair of equipotentials $\beta_{0}$ in $B(\widehat{z_{0}})$ and $\beta_{1}$ in $B(\widehat{z_{1}})$, we have
$$\frac{1}{d_{3}}\modulus(A(\beta_{1},\beta_{0}))\leqslant\frac{1}{d_{3}}(\modulus(A(\alpha_{0},\beta_{0}))+\modulus(A(\alpha_{1},\beta_{1}))+C).$$

Now consider the following linear system of inequations with real unknowns $x_{0},x_{1},x_{2},x_{3}$.
\begin{equation}\label{EqSysIneq2}
	\left\{ \begin{array}{ccl}
		\dfrac{1}{d_{0}}x_{1} & < & x_{0} \\
		\dfrac{1}{d_{1}}x_{2} & < & x_{1} \\
		\dfrac{1}{d_{2}}x_{0} + \dfrac{1}{d_{2}}x_{3} & < & x_{2} \\
		\dfrac{1}{d_{3}}(x_{0}+x_{1}+C) & < & x_{3} \\
	\end{array}\right.
\end{equation}
Using the transition matrix $M$ coming from Definition \ref{DefTransitionMatrix}, this system is equivalent to
$$MX+\left(\begin{array}{c}0\\0\\0\\\frac{C}{d_{3}}\end{array}\right)<X\quad\text{where}\ X=\left(\begin{array}{c}x_{0}\\x_{1}\\x_{2}\\x_{3}\end{array}\right).$$
Recall that assumption (\ref{H2}) states that the leading eigenvalue $\lambda(\HH_{P},w)$ of $M$ is less than $1$. It follows from Perron-Frobenius Theorem and continuity of spectral radius the existence of a vector $V\in\R^{4}$ with positive entries such that $MV<V$. Now taking $\mu>0$ large enough (for instance $\mu=(\frac{C}{d_{3}}+1)(v_{3}-\frac{1}{d_{3}}v_{0}-\frac{1}{d_{3}}v_{1})^{-1}$), the vector $X=\mu V$ with positive entries solves the linear system of inequations (\ref{EqSysIneq2}).

The equipotentials $\beta_{0},\beta_{1},\beta_{2}$ are uniquely defined by
$$\frac{1}{2\pi}\log\left(\frac{1}{L_{k}(\beta_{k})}\right)=\modulus(A(\alpha_{k},\beta_{k}))=x_{k}\quad\text{for every}\ k\in\{0,1,2\}.$$

For $\beta_{3}^{+}$, choose an arbitrary equipotential in $B(\widehat{z_{0}})$ such that
$$L_{0}(\beta_{0})>L_{0}(\beta_{3}^{+})\ \text{and}\ \frac{1}{2\pi}\log\left(\frac{L_{0}(\beta_{0})}{L_{0}(\beta_{3}^{+})}\right)=\modulus(A(\beta_{0},\beta_{3}^{+}))>1.$$

 Then $\beta_{3}^{-}$ is uniquely defined by
$$L_{0}(\beta_{3}^{+})>L_{0}(\beta_{3}^{-})\ \text{and}\ \frac{1}{2\pi}\log\left(\frac{L_{0}(\beta_{3}^{+})}{L_{0}(\beta_{3}^{-})}\right)=\modulus(A(\beta_{3}^{+},\beta_{3}^{-}))=x_{3}.$$

It follows from construction that $\beta_{0}$, $\beta_{1}$, $\beta_{2}$, $\beta_{3}^{+}$, and $\beta_{3}^{-}$ satisfy all the requirements of Lemma \ref{LemEquip}, the fourth inequality in linear system (\ref{EqSysIneq}) coming from the last inequality in linear system (\ref{EqSysIneq2}) and Lemma \ref{LemGrotzch}.
\end{proof}

It turns out in the proof above that the lower bound of the last inequality in linear system (\ref{EqSysIneq}) may be changed for any positive constant (which depends only on the integers $d_{0}$, $d_{1}$, $d_{2}$, and $d_{3}$). As we will see later in Lemma \ref{LemFolding}, the lower bound $1$ ensures sufficient space to make the surgery in $A(\beta_{0},\beta_{3}^{+})$. However, the author guesses that the last inequality in linear system (\ref{EqSysIneq}) is not necessary (see discussion after the proof of Lemma \ref{LemAnnulusDisk}).

The system of equipotentials coming from Lemma \ref{LemEquip} will be used to divide $\CC$ into several pieces on which a quasiregular map $F$ will be piecewisely defined. This map $F$ should be carefully defined in such a way that its dynamics is encoded by the weighted dynamical tree $(\HH_{P},w)$ (see Theorem \ref{ThmMain1}).

For instance, the first step of the construction which corresponds to the dynamics on $e_{1}\cup e_{2}$ for $\HH_{P}$ is the following. Denote by $\beta_{0,1}$ the preimage of $\beta_{1}$ in $B(\widehat{z_{0}})$ (see Figure \ref{FigEquipotentials}). From the first inequality in linear system (\ref{EqSysIneq}), $\beta_{0,1}$ is an equipotential of level $L_{0}(\beta_{0,1})>L_{0}(\beta_{0})$. Denote by $D(\beta_{0,1})$ the open disk bounded by $\beta_{0,1}$ and containing $\{\widehat{z_{1}},\widehat{z_{2}}\}$ (and hence $J(\widehat{f})\cup B(\widehat{z_{1}})\cup B(\widehat{z_{2}})$ as well). Then $F$ is defined to be the rational map $\widehat{f}$ on $D(\beta_{0,1})$. Remark that $F|_{D(\beta_{0,1})}$ continuously extends to $\beta_{0,1}$ by a degree $d_{0}$ covering denoted by $F|_{\beta_{0,1}}:\beta_{0,1}\rightarrow\beta_{1}$.

\subsection{Folding with an annulus-disk surgery}

The aim of this part of the construction is to realize the folding corresponding to the critical point $c_{0}$ in $\HH_{P}$. More precisely $F$ should holomorphically maps a small annulus (corresponding to a neighborhood of $c_{0}$ in $\HH_{P}$) onto a disk (corresponding to a neighborhood of $c_{1}$ in $\HH_{P}$) with respect to the degrees $d_{0},d_{3}$.

Let $\gamma_{1}$ be an arbitrary equipotential in $B(\widehat{z_{1}})$ such that $L_{1}(\gamma_{1})<L_{1}(\beta_{1})$. Denote by $D(\gamma_{1})$ the open disk bounded by $\gamma_{1}$ and containing $\widehat{z_{1}}$. In order to follow more easily the construction, we will slightly improve the notation. So let $\gamma_{0,1}$ be the equipotential $\beta_{0}$, keeping in mind that $\gamma_{0,1}$ will be mapped onto $\gamma_{1}$ by a degree $d_{0}$ covering. Notice that the first inequality in linear system (\ref{EqSysIneq}) of Lemma \ref{LemEquip} implies $L_{0}(\beta_{0,1})>L_{0}(\gamma_{0,1})$. Similarly let $\beta_{3,1}$ be the equipotential $\beta_{3}^{+}$, keeping in mind that $\beta_{3,1}$ will be mapped onto $\beta_{1}$ by a degree $d_{3}$ covering.

\begin{lem}\label{LemFolding}
There exist an equipotential $\gamma_{3,1}$ in $B(\widehat{z_{0}})$ and a holomorphic branched covering $F|_{A(\gamma_{0,1},\gamma_{3,1})}:A(\gamma_{0,1},\gamma_{3,1})\rightarrow D(\gamma_{1})$ such that:
\begin{description}
	\item[(i)] $L_{0}(\beta_{0,1})>L_{0}(\gamma_{0,1})>L_{0}(\gamma_{3,1})>L_{0}(\beta_{3,1})$;
	\item[(ii)] $F|_{A(\gamma_{0,1},\gamma_{3,1})}$ has degree $d_{0}+d_{3}$ and has $d_{0}+d_{3}$ critical points counted with multiplicity, which one of them, denoted by $c$, satisfies $F|_{A(\gamma_{0,1},\gamma_{3,1})}(c)=\widehat{z_{1}}$;
	\item[(iii)] $F|_{A(\gamma_{0,1},\gamma_{3,1})}$ continuously extends to $\gamma_{0,1}\cup\gamma_{3,1}$ by a degree $d_{0}$ covering $F|_{\gamma_{0,1}}:\gamma_{0,1}\rightarrow\gamma_{1}$ and a degree $d_{3}$ covering $F|_{\gamma_{3,1}}:\gamma_{3,1}\rightarrow\gamma_{1}$.
\end{description}
\end{lem}

\begin{proof}
Let $G:A(\gamma,\gamma')\rightarrow\D$ be a holomorphic branched covering coming from Lemma \ref{LemAnnulusDisk} for the integers $n=d_{0}$ and $n'=d_{3}$. Define the equipotential $\gamma_{3,1}$ by
$$L_{0}(\gamma_{0,1})>L_{0}(\gamma_{3,1})\ \text{and}\ \frac{1}{2\pi}\log\left(\frac{L_{0}(\gamma_{0,1})}{L_{0}(\gamma_{3,1})}\right)=\modulus(A(\gamma_{0,1},\gamma_{3,1}))=\modulus(A(\gamma,\gamma')).$$

Since $\modulus(A(\gamma_{0,1},\beta_{3,1}))=\modulus(A(\beta_{0},\beta_{3}^{+}))>1$ (from the last inequality in linear system (\ref{EqSysIneq}) of Lemma \ref{LemEquip}) and $\modulus(A(\gamma_{0,1},\gamma_{3,1}))=\modulus(A(\gamma,\gamma'))\leqslant 1$ (from the point \textbf{(iii)} in Lemma \ref{LemAnnulusDisk}), it follows that $L_{0}(\gamma_{3,1})>L_{0}(\beta_{3,1})$ and the point \textbf{(i)} holds.

Now let $\psi$ be any biholomorphic map from $A(\gamma_{0,1},\gamma_{3,1})$ onto $A(\gamma,\gamma')$. The existence of such a biholomorphic map is ensured by the fact that these two open annuli have same modulus. Since $A(\gamma_{0,1},\gamma_{3,1})$ and $A(\gamma,\gamma')$ are bounded by quasicircles, $\psi$ may be continuously extended to $\gamma_{0,1}\cup\gamma_{3,1}$ by two homeomorphisms.

Let $c$ be the preimage under $\psi$ of any critical point of $G$ and let $\phi:\D\rightarrow D(\gamma_{1})$ be any Riemann mapping of $D(\gamma_{1})$ such that $\phi(G(\psi(c)))=\widehat{z_{1}}$. Since $D(\gamma_{1})$ is bounded by an equipotential, $\phi$ may be continuously extended to $\partial\D$ by a homeomorphism.

Then $F|_{A(\gamma_{0,1},\gamma_{3,1})}=\phi\circ G\circ\psi$ is holomorphic on $A(\gamma_{0,1},\gamma_{3,1})$ and satisfies \textbf{(ii)}, and \textbf{(iii)} by construction.
\end{proof}

Figure \ref{FigFolding} depicts the map $F|_{A(\gamma_{0,1},\gamma_{3,1})}$ coming from Lemma \ref{LemFolding}.

\begin{figure}[!htb]
	\begin{center}
	\includegraphics[width=\textwidth]{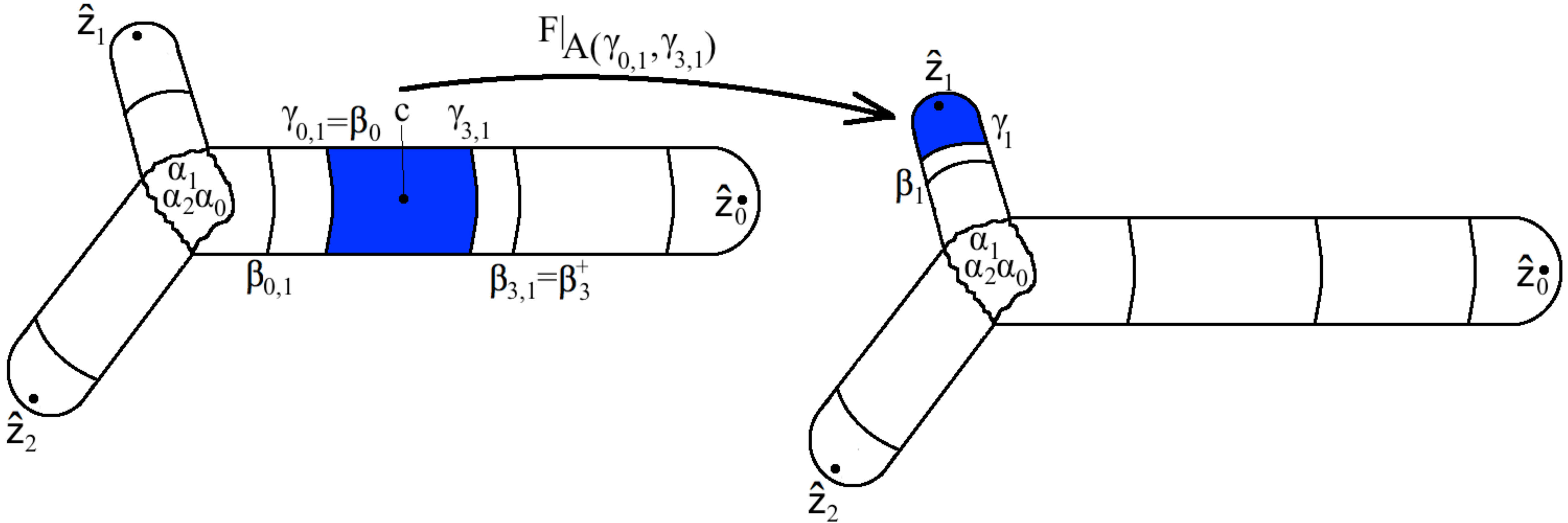}
	\caption{The map $F|_{A(\gamma_{0,1},\gamma_{3,1})}$ coming from Lemma \ref{LemFolding} displayed on the Riemann sphere which is topologically distorted to emphasize the three domains $B(\widehat{z_{0}})$, $B(\widehat{z_{1}})$, and $B(\widehat{z_{2}})$ (compare with Figure \ref{FigPersianCarpetTree}.c).}
	\label{FigFolding}
	\end{center}
\end{figure}

\subsection{Preimage of the branching part}

According to the last two sections, the map $F$ is defined up to there on the union of the open disk $D(\beta_{0,1})$ containing $\{\widehat{z_{1}},\widehat{z_{2}}\}$ with the open annulus $A(\gamma_{0,1},\gamma_{3,1})$ containing $c$. Moreover $F$ maps $c$ to $\widehat{z_{1}}$, $\widehat{z_{1}}$ to $\widehat{z_{2}}$ and $\widehat{z_{2}}$ to $\widehat{z_{0}}$. Now we need to define $F$ near $\widehat{z_{0}}$ by sending $\widehat{z_{0}}$ to $c$ in order to realize a cycle of period $4$ as required in Theorem \ref{ThmMain1}. This should be done carefully so that the quasiconformal surgery may be concluded.

The first problem is that some preimage of $J(\widehat{f})$ (or more precisely of the open annulus $A(\beta_{1},\beta_{0})$ containing $J(\widehat{f})$) must appear in $B(\widehat{z_{0}})$ (compare with Figure \ref{FigPersianCarpetTree}.c where the edge $e_{3}=[c_{0},c_{3}]_{\HH_{P}}$ contains a preimage of the branching point $\alpha$). This is done in Lemma \ref{LemPreimage} below which essentially uses the fourth inequality in linear system (\ref{EqSysIneq}) of Lemma \ref{LemEquip}.

\begin{lem}\label{LemPreimage}
There exist an equipotential $\beta_{3,0}$ in $B(\widehat{z_{0}})$ and a holomorphic covering $F|_{A(\beta_{3,1},\beta_{3,0})}:A(\beta_{3,1},\beta_{3,0})\rightarrow A(\beta_{1},\beta_{0})$ such that:
\begin{description}
	\item[(i)] $L_{0}(\beta_{3,1})>L_{0}(\beta_{3,0})>L_{0}(\beta_{3}^{-})$;
	\item[(ii)] $F|_{A(\beta_{3,1},\beta_{3,0})}$ has degree $d_{3}$ and has no critical point;
	\item[(iii)] $F|_{A(\beta_{3,1},\beta_{3,0})}$ continuously extends to $\beta_{3,1}\cup\beta_{3,0}$ by two degree $d_{3}$ coverings $F|_{\beta_{3,1}}:\beta_{3,1}\rightarrow\beta_{1}$ and $F|_{\beta_{3,0}}:\beta_{3,0}\rightarrow\beta_{0}$.
\end{description}
\end{lem}

\begin{proof}
Define the equipotential $\beta_{3,0}$ by
$$L_{0}(\beta_{3,1})>L_{0}(\beta_{3,0})\ \text{and}\ \frac{1}{2\pi}\log\left(\frac{L_{0}(\beta_{3,1})}{L_{0}(\beta_{3,0})}\right)=\modulus(A(\beta_{3,1},\beta_{3,0}))=\frac{1}{d_{3}}\modulus(A(\beta_{1},\beta_{0})).$$

Since $\modulus(A(\beta_{3,1},\beta_{3,0}))=\frac{1}{d_{3}}\modulus(A(\beta_{1},\beta_{0}))<\modulus(A(\beta_{3}^{+},\beta_{3}^{-}))=\modulus(A(\beta_{3,1},\beta_{3}^{-}))$ (from the fourth inequality in linear system (\ref{EqSysIneq}) of Lemma \ref{LemEquip}), it follows that $L_{0}(\beta_{3,0})>L_{0}(\beta_{3}^{-})$ and the point \textbf{(i)} holds.

Now let $\psi$ be any biholomorphic map from $A(\beta_{3,1},\beta_{3,0})$ onto a round annulus of the form $A_{r}=\{z\in\C\,/\,r<|z|<1\}$ where $r$ is defined by
$$\frac{1}{2\pi}\log\left(\frac{1}{r}\right)=\modulus(A_{r})=\modulus(A(\beta_{3,1},\beta_{3,0})).$$
Since $A(\beta_{3,1},\beta_{3,0})$ is bounded by equipotentials, $\psi$ may be continuously extended to $\beta_{3,1}\cup\beta_{3,0}$ by two homeomorphisms which send $\beta_{3,1}$ onto $\{z\in\C\,/\,|z|=1\}$ and $\beta_{3,0}$ onto $\{z\in\C\,/\,|z|=r\}$.

Similarly, let $\Psi$ be any biholomorphic map from the round annulus $A_{r^{d_{3}}}$ onto $A(\beta_{1},\beta_{0})$. The existence of such a biholomorphic map is ensured by the fact that
$$\modulus(A_{r^{d_{3}}})=\frac{1}{2\pi}\log\left(\frac{1}{r^{d_{3}}}\right)=\frac{d_{3}}{2\pi}\log\left(\frac{1}{r}\right)=d_{3}\modulus(A(\beta_{3,1},\beta_{3,0}))=\modulus(A(\beta_{1},\beta_{0})).$$
Since $A(\beta_{1},\beta_{0})$ is bounded by equipotentials, $\Psi$ may be continuously extended to $\partial A_{r^{d_{3}}}$ by two homeomorphisms which send $\{z\in\C\,/\,|z|=1\}$ onto $\beta_{1}$ and $\{z\in\C\,/\,|z|=r^{d_{3}}\}$ onto $\beta_{0}$.

Then $F|_{A(\beta_{3,1},\beta_{3,0})}=\Psi\circ(z\mapsto z^{d_{3}})\circ\psi$ is holomorphic on $A(\beta_{3,1},\beta_{3,0})$ and satisfies \textbf{(ii)}, and \textbf{(iii)} by construction.
\end{proof}

Figure \ref{FigPreimage} depicts the map $F|_{A(\beta_{3,1},\beta_{3,0})}$ coming from Lemma \ref{LemPreimage}.

\begin{figure}[!htb]
	\begin{center}
	\includegraphics[width=\textwidth]{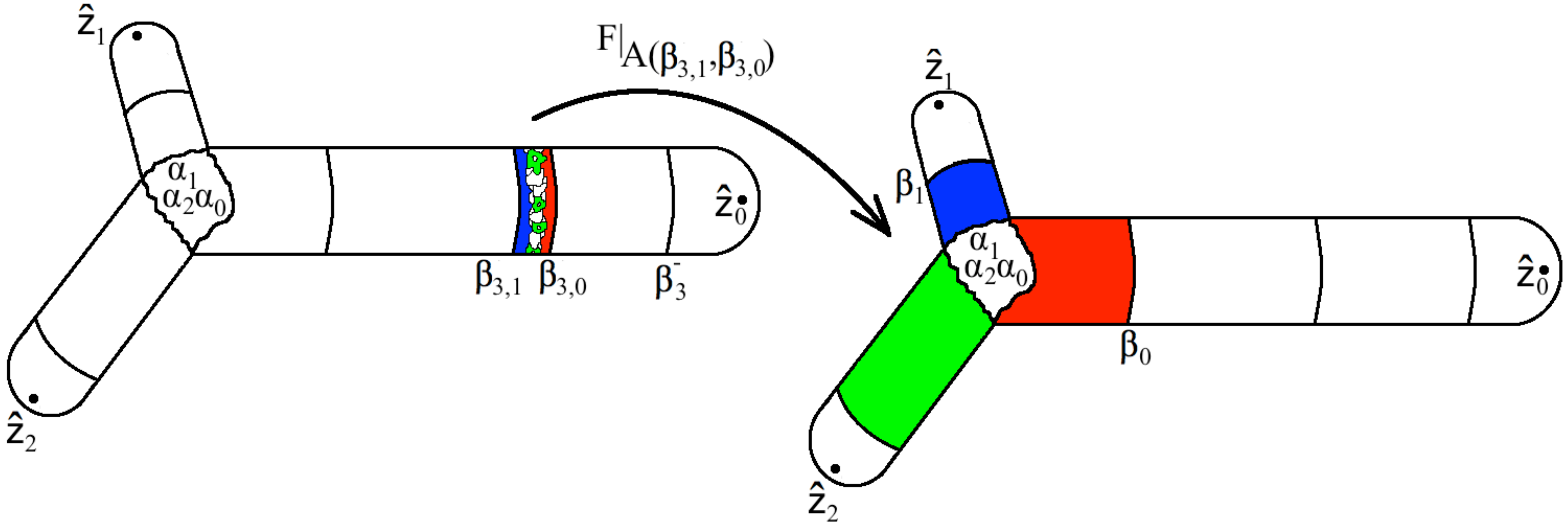}
	\caption{The map $F|_{A(\beta_{3,1},\beta_{3,0})}$ coming from Lemma \ref{LemPreimage} displayed on the Riemann sphere which is topologically distorted to emphasize the three domains $B(\widehat{z_{0}})$, $B(\widehat{z_{1}})$, and $B(\widehat{z_{2}})$ (compare with Figure \ref{FigPersianCarpetTree}.c).}
	\label{FigPreimage}
	\end{center}
\end{figure}

\subsection{Achievement of the super-attracting cycle of period $4$}

Now we achieve the definition of $F$ near $\widehat{z_{0}}$. This is done in two parts. Firstly Lemma \ref{LemLastStep1} realizes a preimage of a neighborhood of $\widehat{z_{0}}$ in $B(\widehat{z_{0}})$. Then Lemma \ref{LemLastStep2} defines $F$ near $\widehat{z_{0}}$ by sending a neighborhood of $\widehat{z_{0}}$ onto a neighborhood of $c$ (mapping $\widehat{z_{0}}$ to $c$).

Let $\gamma_{0}$ be an arbitrary equipotential in $B(\widehat{z_{0}})$ such that $L_{0}(\beta_{0})=L_{0}(\gamma_{0,1})>L_{0}(\gamma_{0})>L_{0}(\gamma_{3,1})$ and $A(\gamma_{0},\gamma_{3,1})$ contains the critical point $c$.

\begin{lem}\label{LemLastStep1}
There exist two equipotentials $\gamma_{3,0}$ and $\delta_{3,c}^{+}$ in $B(\widehat{z_{0}})$, a quasicircle $\delta_{c}^{+}$ in $A(\gamma_{0},\gamma_{3,1})$ which separates $c$ from $\gamma_{0}\cup\gamma_{3,1}$, and a holomorphic covering $F|_{A(\gamma_{3,0},\delta_{3,c}^{+})}:A(\gamma_{3,0},\delta_{3,c}^{+})\rightarrow A(\gamma_{0},\delta_{c}^{+})$ such that:
\begin{description}
	\item[(i)] $L_{0}(\beta_{3,0})>L_{0}(\gamma_{3,0})>L_{0}(\delta_{3,c}^{+})>L_{0}(\beta_{3}^{-})$;
	\item[(ii)] $F|_{A(\gamma_{3,0},\delta_{3,c}^{+})}$ has degree $d_{3}$ and has no critical point;
	\item[(iii)] $F|_{A(\gamma_{3,0},\delta_{3,c}^{+})}$ continuously extends to $\gamma_{3,0}\cup\delta_{3,c}^{+}$ by two degree $d_{3}$ coverings $F|_{\gamma_{3,0}}:\gamma_{3,0}\rightarrow\gamma_{0}$ and $F|_{\delta_{3,c}^{+}}:\delta_{3,c}^{+}\rightarrow\delta_{c}^{+}$.
\end{description}
\end{lem}

\begin{proof}
Applying Lemma \ref{LemConfGeomEx}, we get a quasicircle $\delta_{c}^{+}$ in $A(\gamma_{0},\gamma_{3,1})$ which separates $c$ from $\gamma_{0}\cup\gamma_{3,1}$ and such that
$$\frac{1}{d_{3}}\modulus(A(\gamma_{0},\delta_{c}^{+}))<\modulus(A(\beta_{3,0},\beta_{3}^{-})).$$
Therefore we can find two equipotentials $\gamma_{3,0}$ and $\delta_{3,c}^{+}$ in $B(\widehat{z_{0}})$ so that
\begin{eqnarray*}
	&& L_{0}(\beta_{3,0})>L_{0}(\gamma_{3,0})>L_{0}(\delta_{3,c}^{+})>L_{0}(\beta_{3}^{-}) \\
	& \text{and}\ & \frac{1}{2\pi}\log\left(\frac{L_{0}(\gamma_{3,0})}{L_{0}(\delta_{3,c}^{+})}\right)=\modulus(A(\gamma_{3,0},\delta_{3,c}^{+}))=\frac{1}{d_{3}}\modulus(A(\gamma_{0},\delta_{c}^{+})).
\end{eqnarray*}
The point \textbf{(i)} holds by definition. For the two other points, the proof may be achieved as that one of Lemma \ref{LemPreimage}.
\end{proof}

Figure \ref{FigLastStep} depicts the equipotentials involved in Lemma \ref{LemLastStep1} and the map $F|_{A(\gamma_{3,0},\delta_{3,c}^{+})}$.

\begin{figure}[!htb]
	\begin{center}
	\includegraphics[width=\textwidth]{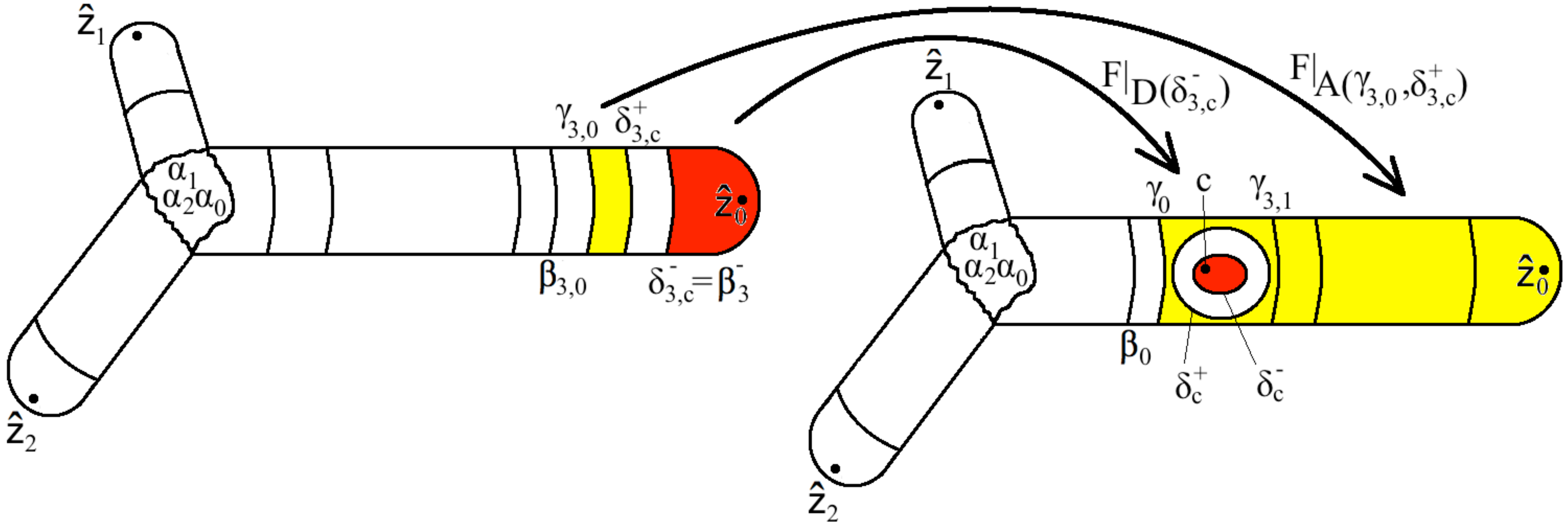}
	\caption{The maps $F|_{A(\gamma_{3,0},\delta_{3,c}^{+})}$ and $F|_{D(\delta_{3,c}^{-})}$ coming from Lemma \ref{LemLastStep1} and Lemma \ref{LemLastStep2} displayed on the Riemann sphere which is topologically distorted to emphasize the three domains $B(\widehat{z_{0}})$, $B(\widehat{z_{1}})$, and $B(\widehat{z_{2}})$ (compare with Figure \ref{FigPersianCarpetTree}.c).} \label{FigLastStep}
	\end{center}
\end{figure}

It remains to define $F$ near $\widehat{z_{0}}$. Let $\delta_{c}^{-}$ be an arbitrary quasicircle which separates $c$ from $\delta_{c}^{+}$. We slightly improve the notation by denoting $\delta_{3,c}^{-}$ the equipotential $\beta_{3}^{-}$ keeping in mind that $\delta_{3,c}^{-}$ will be mapped onto $\delta_{c}^{-}$ by a degree $d_{3}$ covering. Finally, denote by $D(\delta_{3,c}^{-})$ the open disk bounded by $\delta_{3,c}^{-}$ and containing $\widehat{z_{0}}$, and by $D(\delta_{c}^{-})$ the open disk bounded by $\delta_{c}^{-}$ and containing $c$.

\begin{lem}\label{LemLastStep2}
There exists a holomorphic branched covering $F|_{D(\delta_{3,c}^{-})}:D(\delta_{3,c}^{-})\rightarrow D(\delta_{c}^{-})$ such that:
\begin{description}
	\item[(i)] $F|_{D(\delta_{3,c}^{-})}$ has degree $d_{3}$ and has only one critical point which is $\widehat{z_{0}}$ with $F|_{D(\delta_{3,c}^{-})}(\widehat{z_{0}})=c$;
	\item[(ii)] $F|_{D(\delta_{3,c}^{-})}$ continuously extends to $\delta_{3,c}^{-}$ by a degree $d_{3}$ covering $F|_{\delta_{3,c}^{-}}:\delta_{3,c}^{-}\rightarrow\delta_{c}^{-}$.
\end{description}
\end{lem}

\begin{proof}
Let $\phi:\D\rightarrow D(\delta_{3,c}^{-})$ be any Riemann mapping of $D(\delta_{3,c}^{-})$ such that $\phi(0)=\widehat{z_{0}}$, and let $\Phi:\D\rightarrow D(\delta_{c}^{-})$ be any Riemann mapping of $D(\delta_{c}^{-})$ such that $\Phi(0)=c$. Since $D(\delta_{3,c}^{-})$ and $D(\delta_{c}^{-})$ are bounded by quasicircles, $\phi$ and $\Phi$ may be continuously extended to $\partial\D$ by homeomorphisms.

Then $F|_{D(\delta_{3,c}^{-})}=\Phi\circ(z\mapsto z^{d_{3}})\circ\phi^{-1}$ gives the result.
\end{proof}

Figure \ref{FigLastStep} depicts the map $F|_{D(\delta_{3,c}^{-})}$ coming from Lemma \ref{LemLastStep2}.

\subsection{Uniformization}

At first we sum up in the following table the definition of $F$ up to there.

\begin{center}
\begin{tabular}{|c|c|c|c|c|}
	\hline
	domains & images & {\begin{tabular}{c} cont. extensions \\ on boundaries \end{tabular}} & {\begin{tabular}{c} critical points \\ with multiplicity \end{tabular}} & critical values \\
	\hline
	\hline
	\rule[-10pt]{0pt}{30pt} $D(\beta_{0,1})$ & $\CC$ & $\beta_{0,1}\longlong{d_{0}:1}\beta_{1}$ & {$\begin{array}{c} \widehat{z_{1}}\ \text{with mult.}\ d_{1}-1 \\ \widehat{z_{2}}\ \text{with mult.}\ d_{2}-1 \end{array}$} & {$\begin{array}{c} F(\widehat{z_{1}})=\widehat{z_{2}} \\ F(\widehat{z_{2}})=\widehat{z_{0}} \end{array}$} \\
	\hline
	\rule[-10pt]{0pt}{30pt} $A(\gamma_{0,1},\gamma_{3,1})$ & $D(\gamma_{1})$ & {$\begin{array}{c} \gamma_{0,1}\longlong{d_{0}:1}\gamma_{1} \\ \gamma_{3,1}\longlong{d_{3}:1}\gamma_{1} \end{array}$} & {$\begin{array}{c} c\in\{d_{0}+d_{3}\ \text{crit. pts} \\ \text{counted with mult.}\} \end{array}$} & {$\begin{array}{c} F(c)=\widehat{z_{1}} \\ \text{and others} \end{array}$} \\
	\hline
	\rule[-10pt]{0pt}{30pt} $A(\beta_{3,1},\beta_{3,0})$ & $A(\beta_{1},\beta_{0})$ & {$\begin{array}{c} \beta_{3,1}\longlong{d_{3}:1}\gamma_{1} \\ \beta_{3,0}\longlong{d_{3}:1}\beta_{0} \end{array}$} & $\emptyset$ & $\emptyset$ \\
	\hline
	\rule[-10pt]{0pt}{30pt} $A(\gamma_{3,0},\delta_{3,c}^{+})$ & $A(\gamma_{0},\delta_{c}^{+})$ & {$\begin{array}{c} \gamma_{3,0}\longlong{d_{3}:1}\gamma_{0} \\ \delta_{3,c}^{+}\longlong{d_{3}:1}\delta_{c}^{+} \end{array}$} & $\emptyset$ & $\emptyset$ \\
	\hline
	\rule[-10pt]{0pt}{30pt} $D(\delta_{3,c}^{-})$ & $D(\delta_{c}^{-})$ & $\delta_{3,c}^{-}\longlong{d_{3}:1}\delta_{c}^{-}$ & $\widehat{z_{0}}$ with mult. $d_{3}-1$ & $F(\widehat{z_{0}})=c$ \\
	\hline
\end{tabular}% ATTENTION : bidouilles (\rule) pour agrandir les cases
\end{center}

So $F$ is holomorphically defined on $H=D(\beta_{0,1})\cup A(\gamma_{0,1},\gamma_{3,1})\cup A(\beta_{3,1},\beta_{3,0})\cup A(\gamma_{3,0},\delta_{3,c}^{+})\cup D(\delta_{3,c}^{-})$ with continuous extension on the boundary. It remains to define $F$ on the complement $Q=\CC-\overline{H}=A(\beta_{0,1},\gamma_{0,1})\cup A(\gamma_{3,1},\beta_{3,1})\cup A(\beta_{3,0},\gamma_{3,0})\cup A(\delta_{3,c}^{+},\delta_{3,c}^{-})$. This is done in the following lemma.

\begin{lem}\label{LemQcExtension}
The map $F|_{\overline{H}}:\overline{H}\rightarrow\CC$ extends to a quasiregular map $F:\CC\rightarrow\CC$ by quasiconformal coverings defined on each connected component of $Q=\CC-\overline{H}$.

Moreover there exists an open subset $E\subset H$ such that $F(E)\subset E$ and $F^{2}(\overline{Q})\subset E$.
\end{lem}

In particular, notice that the quasiregular map $F:\CC\rightarrow\CC$ has no more critical points than those coming from the holomorphic restriction $F|_{H}:H\rightarrow\CC$.

\begin{proof}
Remark that every connected component of $Q$ is an open annulus whose boundary is the disjoint union of two quasicircles where $F$ realizes two coverings of same degree (and same orientation). By interpolation, $F$ may be continuously extended to each connected component of $Q$ by a covering of degree corresponding to that one on the boundary. Since all the connected components of the boundary of $Q$, together with their images by $F$, are quasicircles, each interpolation may be carefully done in such a way that the resulting map is actually quasiconformal on the Riemann sphere. In short, $F$ quasiregularly extends to $Q$ by
\begin{itemize}
	\item a degree $d_{0}$ quasiconformal covering $F|_{A(\beta_{0,1},\gamma_{0,1})}:A(\beta_{0,1},\gamma_{0,1})\rightarrow A(\beta_{1},\gamma_{1})$;
	\item a degree $d_{3}$ quasiconformal covering $F|_{A(\gamma_{3,1},\beta_{3,1})}:A(\gamma_{3,1},\beta_{3,1})\rightarrow A(\gamma_{1},\beta_{1})$;
	\item a degree $d_{3}$ quasiconformal covering $F|_{A(\beta_{3,0},\gamma_{3,0})}:A(\beta_{3,0},\gamma_{3,0})\rightarrow A(\beta_{0},\gamma_{0})$;
	\item a degree $d_{3}$ quasiconformal covering $F|_{A(\delta_{3,c}^{+},\delta_{3,c}^{-})}:A(\delta_{3,c}^{+},\delta_{3,c}^{-})\rightarrow A(\delta_{c}^{+},\delta_{c}^{-})$.
\end{itemize}
In particular, we have $F(Q)=A(\beta_{1},\gamma_{1})\cup A(\beta_{0},\gamma_{0})\cup A(\delta_{c}^{+},\delta_{c}^{-})$ (see figure \ref{FigFinalDynamics} to follow the continuation of the proof).

Now denote by $\beta_{1,2}$ the preimage of $\beta_{2}$ in $B(\widehat{z_{1}})$ under $F$ (thus under $\widehat{f}$) and similarly by $\beta_{2,3}^{-}$ the preimage of $\beta_{3}^{-}$ in $B(\widehat{z_{2}})$ (see Figure \ref{FigEquipotentials}). Moreover denote by $D(\beta_{1,2})$ the open disk bounded by $\beta_{1,2}$ and containing $\widehat{z_{1}}$, and by $D(\beta_{2,3}^{-})$ the open disk bounded by $\beta_{2,3}^{-}$ and containing $\widehat{z_{2}}$. Finally, let $E$ be the union $D(\beta_{1,2})\cup D(\beta_{2,3}^{-})\cup D(\delta_{3,c}^{-})\cup A(\gamma_{0,1},\gamma_{3,1})$.

At first remark that $E$ is an open subset of $H=D(\beta_{0,1})\cup A(\gamma_{0,1},\gamma_{3,1})\cup A(\beta_{3,1},\beta_{3,0})\cup A(\gamma_{3,0},\delta_{3,c}^{+})\cup D(\delta_{3,c}^{-})$. Indeed we have $\overline{D(\beta_{1,2})}\cup\overline{D(\beta_{2,3}^{-})}\subset D(\beta_{0,1})$ from definition of $D(\beta_{0,1})$.

Moreover, it follows from definition of $F$ on $H$ that $F(E)=D(\beta_{2})\cup D(\beta_{3}^{-})\cup D(\delta_{c}^{-})\cup D(\gamma_{1})$ where $D(\beta_{2})$ denotes the open disk bounded by $\beta_{2}$ and containing $\widehat{z_{2}}$, and $D(\beta_{3}^{-})=D(\delta_{3,c}^{-})$ is the open disk bounded by $\beta_{3}^{-}=\delta_{3,c}^{-}$ and containing $\widehat{z_{0}}$.

Furthermore, according to the whole construction, we have
\begin{itemize}
	\item from Lemma \ref{LemEquip} and definition of $\gamma_{1}$: $\overline{A(\beta_{1},\gamma_{1})}\cup\overline{D(\gamma_{1})}\subset D(\beta_{1,2})$ and $\overline{D(\beta_{2})}\subset D(\beta_{2,3}^{-})$;
	\item from definition of $\gamma_{0}$ and recalling $\beta_{0}=\gamma_{0,1}$: $A(\beta_{0},\gamma_{0})\subset A(\gamma_{0,1},\gamma_{3,1})$;
	\item from definitions of $\delta_{c}^{-}$, $\delta_{c}^{+}$ and $\gamma_{0}$: $\overline{A(\delta_{c}^{+},\delta_{c}^{-})}\cup\overline{D(\delta_{c}^{-})}\subset A(\gamma_{0},\gamma_{3,1})\subset A(\gamma_{0,1},\gamma_{3,1})$.
\end{itemize}

Putting everything together gives the following diagram in which the arrows $\overset{F}{\longrightarrow}$ stand for images under $F$, $\overset{\subset}{\longrightarrow}$ stand for inclusions, $\overset{\subset\subset}{\longrightarrow}$ stand for compact inclusions (namely $A\overset{\subset\subset}{\longrightarrow}B$ if and only if $\overline{A}\subset B$) and $\overset{=}{\longrightarrow}$ stands for equality.

$$\xymatrix @!0 @R=4pc @C=3pc {
	Q \ar[d]^{F} & = & A(\beta_{0,1},\gamma_{0,1}) \ar[d]^{F} & \cup & A(\gamma_{3,1},\beta_{3,1}) \ar[dll]_{F} & \cup & A(\beta_{3,0},\gamma_{3,0}) \ar[dll]_{F} & \cup & A(\delta_{3,c}^{+},\delta_{3,c}^{-}) \ar[dll]_{F} && \\
	F(Q) \ar[d]^{\subset} & = & A(\beta_{1},\gamma_{1}) \ar[d]^{\subset\subset} & \cup & A(\beta_{0},\gamma_{0}) \ar[drrrr]^{\subset} & \cup & A(\delta_{c}^{+},\delta_{c}^{-}) \ar[drr]^{\subset\subset} &&&& \\
	E \ar[d]^{F} & = & D(\beta_{1,2}) \ar[drr]^{F} & \cup & D(\beta_{2,3}^{-}) \ar[drr]^(0.4){F} & \cup & D(\delta_{3,c}^{-}) \ar[drr]^(0.66){F} & \cup & A(\gamma_{0,1},\gamma_{3,1}) \ar@(d,u)[dllllll]_(0.4){F} && \\
	F(E) \ar[d]^{\subset} & = & D(\gamma_{1}) \ar[d]^{\subset\subset} & \cup & D(\beta_{2}) \ar[d]^{\subset\subset} & \cup & D(\beta_{3}^{-}) \ar[d]^{=} & \cup & D(\delta_{c}^{-}) \ar[d]^{\subset\subset} && \\
	E \ar[d]^{\subset} & = & D(\beta_{1,2}) \ar[d]^{\subset\subset} & \cup & D(\beta_{2,3}^{-}) \ar[dll]_{\subset\subset} & \cup & D(\delta_{3,c}^{-}) \ar@(d,u)[drrrr]^(0.66){=} & \cup & A(\gamma_{0,1},\gamma_{3,1}) \ar@(d,u)[dllll]_(0.66){=} && \\
	H & = & D(\beta_{0,1}) & \cup & A(\gamma_{0,1},\gamma_{3,1}) & \cup & A(\beta_{3,1},\beta_{3,0}) & \cup & A(\gamma_{3,0},\delta_{3,c}^{+}) & \cup & D(\delta_{3,c}^{-})
}$$

In particular, we deduce that $F(Q)\subset E$ and $F(E)\subset E\subset H$. Furthermore, following compact inclusions, it turns out that $F^{2}(\overline{Q})\subset E$.
\end{proof}

\begin{figure}[!htb]
	\begin{center}
	\includegraphics[width=\textwidth]{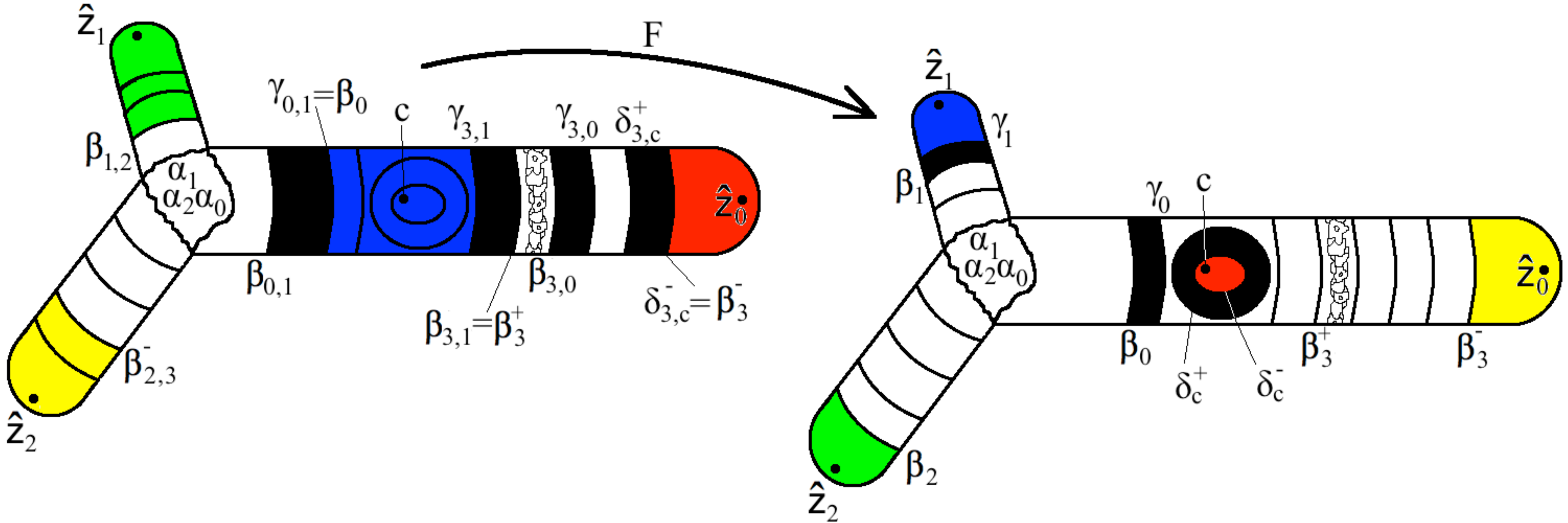}
	\caption{The map $F$ coming from Lemma \ref{LemQcExtension}. On the left topological sphere, the black area stands for $Q$ and the gray area stands for $E$. On the right topological sphere, the black area stands for $F(Q)$ and the gray area stands for $F(E)$.}
	\label{FigFinalDynamics}
	\end{center}
\end{figure}

Now we have a quasiregular map $F$ from the Riemann sphere to itself whose dynamics follows that one of the weighted dynamical tree $(\HH_{P},w)$ (see Figure \ref{FigPersianCarpetTree}.c). We need to find a holomorphic map $f$ conjugated to $F$ so that $f$ follows the same dynamics as well ($f$ should satisfy the requirements of Theorem \ref{ThmMain1} and Theorem \ref{ThmMain2}). To do so, we will apply the Shishikura's fundamental lemma for quasiconformal surgery (stated for the first time in \cite{QuasiconformalSurgery}) that we recall below.

\newpage % ATTENTION : bidouille pour éviter une césure dans le lemme

\begin{lem}[Shishikura's fundamental lemma for quasiconformal surgery]\label{LemQuasiconformalSurgery}
Let $g:\CC\rightarrow\CC$ be a quasiregular map. Assume there are an open set $E\subset\CC$ and an integer $N\geqslant 0$ which satisfy the following conditions:
\begin{itemize}
	\item $g(E)\subset E$;
	\item $g$ is holomorphic on $E$;
	\item $g$ is holomorphic on an open set containing $\CC-g^{-N}(E)$.
\end{itemize}
Then there exists a quasiconformal map $\varphi:\CC\rightarrow\CC$ such that the map $\varphi\circ g\circ\varphi^{-1}$ is holomorphic.
\end{lem}

The result stated in \cite{QuasiconformalSurgery} is a little more general but it easily implies the more explicit statement of Lemma \ref{LemQuasiconformalSurgery} (we refer readers to \cite{QuasiconformalSurgery} and \cite{QuasiconformalSurgeryBook} for a proof and more details).

Here our map $F$ satisfies the three assumptions (indeed $F$ is holomorphic on $H$ hence on $E\subset H$ and Lemma \ref{LemQcExtension} implies that $\CC-F^{-2}(E)\subset\CC-\overline{Q}=H$), so applying Lemma \ref{LemQuasiconformalSurgery} gives a holomorphic map $f:\CC\rightarrow\CC$ quasiconformally conjugated to $F:\CC\rightarrow\CC$ as desired.

\begin{lem}\label{LemDegree}
The rational map $f:\CC\rightarrow\CC$ obtained above has degree $\widehat{d}+d_{3}$ and has a super-attracting cycle $\{z_{0},z_{1},z_{2},z_{3}\}$ of period $4$ which is accumulated by every critical orbit. In particular, $f$ is hyperbolic.
\end{lem}

\begin{proof}
Since $f$ is quasiconformally conjugated to $F$, the critical points of $f$ are images under a quasiconformal map $\varphi$ of the critical points of $F$ with same multiplicities. More precisely, the critical points of $f$ are:
\begin{itemize}
	\item $z_{1}=\varphi(\widehat{z_{1}})\in\varphi(D(\beta_{1,2}))\subset\varphi(E)$ with multiplicity $d_{1}-1$;
	\item $z_{2}=\varphi(\widehat{z_{2}})\in\varphi(D(\beta_{2,3}^{-}))\subset\varphi(E)$ with multiplicity $d_{2}-1$;
	\item $d_{0}+d_{3}$ critical points counted with multiplicity in $\varphi(A(\gamma_{0,1},\gamma_{3,1}))\subset\varphi(E)$, which one of them is given by $z_{0}=\varphi(c)$;
	\item $z_{3}=\varphi(\widehat{z_{0}})\in\varphi(D(\delta_{3,c}^{-}))\subset\varphi(E)$ with multiplicity $d_{3}-1$.
\end{itemize}

According to the Riemann-Hurwitz formula, it follows that the number of critical points counted with multiplicity is given by
$$2\deg(f)-2=(d_{1}-1)+(d_{2}-1)+(d_{0}+d_{3})+(d_{3}-1)$$
and hence
$$\deg(f)=\frac{1}{2}(d_{0}+d_{1}+d_{2}-1)+d_{3}=\widehat{d}+d_{3}.$$

Notice that $\{z_{0},z_{1},z_{2},z_{3}\}$ forms a super-attracting cycle of period $4$. Moreover every critical point of $f$ lies in the forward invariant open set $\varphi(E)$, namely a disjoint union of four open subsets of $\CC$ each containing one point of $\{z_{0},z_{1},z_{2},z_{3}\}$. Consequently, every critical orbit accumulates this super-attracting cycle.
\end{proof}

%*******************************************************************************************

\section{Properties}\label{SecProperties}

The aim of this section is to achieve the proofs of Theorem \ref{ThmMain1} and Theorem \ref{ThmMain2}. More precisely we are going to show that the rational map $f$ constructed in the previous section satisfies all the requirements of these two theorems. Section \ref{SubSecExchangingDynamics} focuses on the dynamical properties of $f$ (stated in Theorem \ref{ThmMain1}), and Section \ref{SubSecTopology} deals with the topological properties of the Julia component of $f$ (stated in Theorem \ref{ThmMain2}).

In order to lighten notations, we forget the quasiconformal map $\varphi$ provided by Lemma \ref{LemQuasiconformalSurgery} to denote the image under $\varphi$ of any set introduced in the previous section (equivalently speaking, we act as if the quasiregular map $F$ constructed in the previous section is actually holomorphic).

\subsection{Exchanging dynamics}\label{SubSecExchangingDynamics}

Consider the following pairwise disjoint open annuli (see Figure \ref{FigEncoding}).
$$A_{0}=A(\alpha_{0},\beta_{0}),\ A_{1}=A(\alpha_{1},\beta_{1}),\ A_{2}=A(\alpha_{2},\beta_{2}), \ \text{and}\ A_{3}=A(\beta_{3}^{+},\beta_{3}^{-}).$$
Then, consider the connected components of the preimage under $f$ of $A_{0}\cup A_{1}\cup A_{2}\cup A_{3}$ which are contained as essential subannulus in one of these open annuli, namely:
\begin{itemize}
	\item $A_{0,1}=A(\alpha_{0},\beta_{0,1})$;
	\item $A_{1,2}=A(\alpha_{1},\beta_{1,2})$;
	\item $A_{2,0}=A(\alpha_{2},\beta_{2,0})$ where $\beta_{2,0}$ is the preimage of $\beta_{0}$ in $B(\widehat{z_{2}})$ (see Figure \ref{FigEquipotentials});
	\item $A_{2,3}=A(\beta_{2,3}^{+},\beta_{2,3}^{-})$ where $\beta_{2,3}^{+}$ is the preimage of $\beta_{3}^{+}$ in $B(\widehat{z_{2}})$ (see Figure \ref{FigEquipotentials});
	\item $A_{3,0}=A(\alpha_{3,0},\beta_{3,0})$ where $\alpha_{3,0}$ is the preimage of $\alpha_{0}$ in $A(\beta_{3,1},\beta_{3,0})$ (see Lemma \ref{LemPreimage});
	\item $A_{3,1}=A(\beta_{3,1},\alpha_{3,1})$ where $\alpha_{3,1}$ is the preimage of $\alpha_{1}$ in $A(\beta_{3,1},\beta_{3,0})$ (see Lemma \ref{LemPreimage}).
\end{itemize}
Notice that the notation is chosen so that each $A_{i,j}$ is contained as essential subannulus in $A_{i}$, and $f|_{A_{i,j}}:A_{i,j}\rightarrow A_{j}$ is a degree $d_{i}$ covering. Remark that some connected components of $f^{-1}(A_{3})$ are included in $A_{3}$ as well (from Lemma \ref{LemLastStep1}, see Figure \ref{FigLastStep}), but none of them is contained in $A_{3}$ as essential subannulus.

\begin{figure}[!htb]
	\begin{center}
	\includegraphics[width=\textwidth]{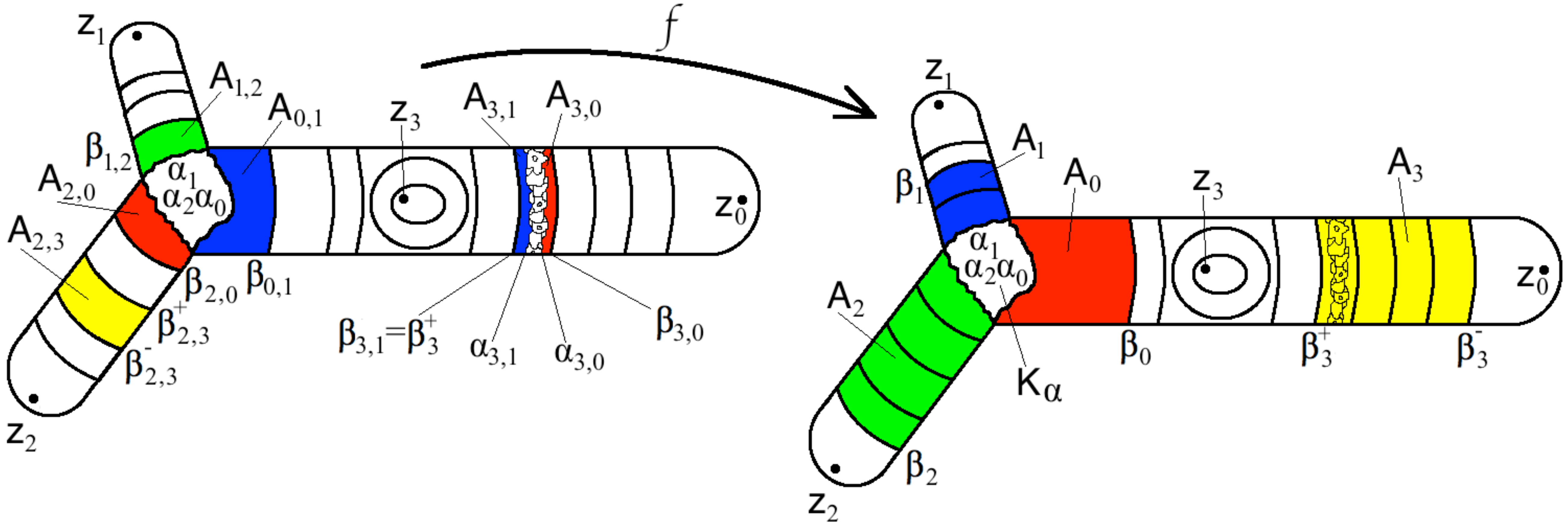}
	\caption{The various annuli considered to encode the exchanging dynamics.}
	\label{FigEncoding}
	\end{center}
\end{figure}

Denote by $\mathcal{A}$ the collection of all connected components of the non-escaping set induced by $f|_{U}:U\rightarrow A_{0}\cup A_{1}\cup A_{2}\cup A_{3}$ on the union of subannuli $U=A_{0,1}\cup A_{1,2}\cup A_{2,0}\cup A_{2,3}\cup A_{3,0}\cup A_{3,1}$.
$$\mathcal{A}=\Big\{J\ \text{connected component of}\ \{z\in U\,/\,\forall n\geqslant 0,\,f^{n}(z)\in U\}\Big\}$$

Let $J_{\alpha}$ be the continuum in $\CC$ which corresponds to the Julia set $J(\widehat{f})$ of $\widehat{f}$ (more precisely, $J_{\alpha}$ is the image of $J(\widehat{f})$ under the quasiconformal map $\varphi$ provided by Lemma \ref{LemQuasiconformalSurgery}). Remark that $J_{\alpha}$ is fixed under iteration of $f$ and $J_{\alpha}$ intersects $\overline{U}$ (along $\alpha_{0}\cup\alpha_{1}\cup\alpha_{2}$). Denote by $\mathcal{A}_{\alpha}$ the collection of all continua which are eventually mapped onto $J_{\alpha}$ and whose every iterate intersects $\overline{U}$.
$$\mathcal{A}_{\alpha}=\Big\{J\ \text{connected component of}\ \bigcup_{n\geqslant 0}f^{-n}(J_{\alpha})\ \text{such that}\ \forall n\geqslant 0,\,f^{n}(J)\cap\overline{U}\neq\emptyset\Big\}$$

Finally, denote by $\mathcal{A}^{\star}$ the union $\mathcal{A}\cup\mathcal{A}_{\alpha}$. As collection of pairwise disjoint continua, $\mathcal{A}^{\star}$ is endowed with the topology coming from the usual distance between continua on the Riemann sphere $\CC$ (equipped with the spherical metric). It turns out that $f$ induced a topological dynamical system on $\mathcal{A}^{\star}$. This dynamical system may be encoded by the weighted dynamical tree $(\HH_{P},w)$ (see Section \ref{SubSecPersianCarpets}) as it is shown in the following lemma.

\begin{lem}\label{LemEncoding}
There exists a homeomorphism $h:\mathcal{A}^{\star}\rightarrow \JJ(\HH_{P})$ such that the following diagram commutes.
$$\xymatrix{
	\mathcal{A}^{\star} \ar[rr]^{\textstyle f} \ar[d]_{\textstyle h} && \mathcal{A}^{\star} \ar[d]^{\textstyle h} \\
	\JJ(\HH_{P}) \ar[rr]_{\textstyle P} && \JJ(\HH_{P})
}$$
Moreover for every $J\in\mathcal{A}$, the restriction map $f|_{J}$ has degree $w(e_{k})=d_{k}$ where $e_{k}$ is the edge of $\HH_{P}$ which contains $h(J)$.
\end{lem}

\begin{proof}
At first, remark there is a subannulus $A_{i,j}$ for some $i,j\in\{0,1,2,3\}$ if and only if the $(i,j)$-entry of the transition matrix $M=(m_{i,j})_{i,j\in\{0,1,2,3\}}$ is non-zero (see Definition \ref{DefTransitionMatrix}). Indeed, recall that the transition matrix is
$$M=\left(\begin{array}{cccc} 0 & \frac{1}{d_{0}} & 0 & 0 \\ 0 & 0 & \frac{1}{d_{1}} & 0 \\ \frac{1}{d_{2}} & 0 & 0 & \frac{1}{d_{2}} \\ \frac{1}{d_{3}} & \frac{1}{d_{3}} & 0 & 0 \end{array}\right).$$
According to this remark, we introduce the subshift of finite type $(\Sigma,\sigma)$ associated to the transition matrix $M$, namely the restriction of the 4-to-1 shift map on the subset of all infinite sequences of digits in $\{0,1,2,3\}$ such that every adjacent pair of entries lies in $\{(0,1),(1,2),(2,0),(2,3),(3,0),(3,1)\}$.
$$\Sigma=\Big\{s=(s_{0},s_{1},s_{2},\dots)\in\{0,1,2,3\}^{\N}\,/\,\forall k\geqslant 0,\,m_{s_{k},s_{k+1}}\neq 0\Big\}$$
$$\sigma:\Sigma\rightarrow\Sigma,s=(s_{0},s_{1},s_{2},\dots)\mapsto\sigma(s)=(s_{1},s_{2},s_{3},\dots)$$
$\Sigma$ is endowed with the topology coming from the following distance, making it a Cantor set.
$$\forall s,s'\in\Sigma,\ d(s,s')=\sum_{k\geqslant 0}\frac{|s_{k}-s'_{k}|}{4^{k}}$$

Let $S_{\alpha}$ be the subset of $\Sigma$ of three infinite sequences of repeating $0,1,2$ digits.
$$S_{\alpha}=\Big\{(0,1,2,0,1,2,0,1,2,\dots),\ (1,2,0,1,2,0,1,2,0,\dots),\ (2,0,1,2,0,1,2,0,1,\dots)\Big\}$$
We shall identify these three sequences in $\Sigma$, and similarly every subset of sequences which are eventually mapped in $S_{\alpha}$ after the same itinerary under $\sigma$. More precisely, let $\sim$ be the equivalence relation on $\Sigma$ defined by
$$\forall s,s'\in\Sigma,\ s\sim s'\iff\exists n\geqslant 0\,/\,\left\{\begin{array}{l} \forall k\in\{0,1,\dots,n\},\, s_{k}=s'_{k} \\ \sigma^{n}(s),\sigma^{n}(s')\in S_{\alpha} \end{array}\right.$$
and let $\Sigma^{\star}$ be the topological quotient space $\Sigma/\sim$. Remark that $\Sigma^{\star}$ is a Cantor set as well for the quotient topology induced by $\sim$. Abusing notations, every equivalence class containing only one infinite sequence $s\in\Sigma$ which is not eventually mapped in $S_{\alpha}$ is still denoted by $s\in\Sigma^{\star}$, and the map induced by the shift map on $\Sigma^{\star}$ is still denoted by $\sigma$.

We are going to show that $(\mathcal{A}^{\star},f)$ is topologically conjugated to $(\Sigma^{\star},\sigma)$. To do so, consider the itinerary map $h_{1}:\mathcal{A}\rightarrow\Sigma^{\star}$ defined by
$$\forall J\in\mathcal{A},\,h_{1}(J)=(s_{0},s_{1},s_{2},\dots)\quad\text{with}\ f^{k}(J)\subset A_{s_{k}}\ \text{for every}\ k\geqslant 0.$$
This map is well defined and injective by definition of $\mathcal{A}$.

To prove that $h_{1}$ extends to a homeomorphism from $\mathcal{A}^{\star}$ to $\Sigma^{\star}$, we first define by induction for every $s=(s_{0},s_{1},s_{2},\dots)\in\Sigma$ an infinite sequence of subannuli $(A_{s_{0},s_{1},\dots,s_{n}})_{n\geqslant 0}$ such that for every $n\geqslant 0$, $A_{s_{0},s_{1},\dots,s_{n}}$ is contained in $A_{s_{0}}$ as essential subannulus, and $f|_{A_{s_{0},s_{1},\dots,s_{n}}}:A_{s_{0},s_{1},\dots,s_{n}}\rightarrow A_{s_{1},s_{2},\dots,s_{n}}$ is a degree $d_{s_{0}}$ covering. Denote by $A_{s}=A_{s_{0},s_{1},s_{2},\dots}$ the limit set $\bigcap_{n\geqslant 0}\overline{A_{s_{0},s_{1},\dots,s_{n}}}$ which is a continuum.

If $s$ is not eventually mapped in $S_{\alpha}$, then $\overline{A_{s_{0},s_{1},\dots,s_{n}}}$ is contained in $U=A_{0,1}\cup A_{1,2}\cup A_{2,0}\cup A_{2,3}\cup A_{3,0}\cup A_{3,1}$ for every $n\geqslant 0$ large enough and thus $A_{s}$ is a connected component of the non-escaping set, that is an element of $\mathcal{A}$. Moreover, $h_{1}(A_{s})=s$ holds from definition of the itinerary map $h_{1}$.

On the contrary, if $s$ is in $S_{\alpha}$, then $A_{s}$ is either $\alpha_{0}$, $\alpha_{1}$ or $\alpha_{2}$, and in particular $A_{s}$ is contained in $J_{\alpha}$. More generally, if $s$ is eventually mapped in $S_{\alpha}$, then $A_{s}$ is contained in a continuum $J$ which is eventually mapped onto $J_{\alpha}$, that is an element of $\mathcal{A}_{\alpha}$. Moreover, for every $s'\in\Sigma$ such that $s'\sim s$, $A_{s'}$ is contained in the same continuum $J\in\mathcal{A}_{\alpha}$.

Therefore $h_{1}$ extends to a bijective map from $\mathcal{A}^{\star}$ to $\Sigma^{\star}$, by associating to $J\in\mathcal{A}_{\alpha}$ the equivalence class $h_{1}(J)\in\Sigma^{*}$ of the itinerary $s=(s_{0}s_{1},s_{2},\dots)\in\Sigma$ of any subcontinuum in $J$ which is eventually mapped into $\alpha_{0}\cup\alpha_{1}\cup\alpha_{2}$. Furthermore, this extension is actually a conjugation between $f$ and $\sigma$.
$$\forall J\in\mathcal{A}^{\star},\,h_{1}(f(J))=\sigma(h_{1}(J))$$

It remains to prove the continuity. Fix $J\in\mathcal{A}^{\star}$ and let $s=(s_{0},s_{1},s_{2},\dots)\in\Sigma$ be a class representative of $h_{1}(J)$. Let $J'$ be another element of $\mathcal{A}^{\star}$ such that some class representative $s'=(s'_{0},s'_{1},s'_{2},\dots)\in\Sigma$ of $h_{1}(J')$ is arbitrary close to $s$. That implies the first $n$ digits of $s$ and $s'$ coincide for arbitrary large $n\geqslant 0$. In particular, $A_{s}$ and $A_{s'}$ are contained in $\overline{A_{s_{0},s_{1},\dots,s_{n}}}$. Remark that $f^{n}|_{A_{s_{0},s_{1},\dots,s_{n}}}:A_{s_{0},s_{1},\dots,s_{n}}\rightarrow A_{s_{n}}$ is a covering of degree $d_{s_{0}}d_{s_{1}}\dots d_{s_{n-1}}$ tending to infinity with $n$ (since assumption (\ref{H2}) implies that at least two of weights $d_{0}$, $d_{1}$, $d_{2}$, and $d_{3}$ are $\geqslant 2$, see Definition \ref{DefTransitionMatrix}). Therefore $A_{s}$ and $A_{s'}$ are contained in an open annulus of arbitrary small modulus. Then, using extremal length (see \cite{ConformalInvariants}), it follows that $A_{s}\subset J$ and $A_{s'}\subset J'$ are arbitrary close, hence $J$ and $J'$ are arbitrary close in $\mathcal{A}^{\star}$. Consequently $h_{1}^{-1}$ is continuous. The continuity of $h_{1}$ follows from a similar argument.

Similarly, we can show that $(\JJ(\HH_{P}),P)$ is topologically conjugated to $(\Sigma^{\star},\sigma)$ by a homeomorphism $h_{2}:\JJ(\HH_{P})\rightarrow\Sigma^{\star}$. Indeed recall that the dynamical tree $\HH_{P}$ is described by a set of four edges $e_{0},e_{1},e_{2},e_{3}$ where $P$ acts as follows (see Section \ref{SubSecPersianCarpets}).
$$\left\{\begin{array}{l}
	P(e_{0})=e_{1} \\
	P(e_{1})=e_{2} \\
	P(e_{2})=e_{0}\cup e_{3} \\
	P(e_{3})=e_{0}\cup e_{1}
\end{array}\right.$$
Thus, we may find four connected open subsets $I_{0}$, $I_{1}$, $I_{2}$, and $I_{3}$ respectively included in $e_{0}$, $e_{1}$, $e_{2}$, and $e_{3}$ together with six connected open subsets $I_{0,1}$, $I_{1,2}$, $I_{2,0}$, $I_{2,3}$, $I_{3,0}$, and $I_{3,1}$ such that:
\begin{itemize}
	\item each $I_{i,j}$ is contained in $I_{i}$ and $P|_{I_{i,j}}:I_{i,j}\rightarrow I_{j}$ is a homeomorphism;
	\item and $\JJ(\HH_{P})=\Big\{z\in V\,/\,\forall n\geqslant 0,\,P^{n}(z)\in V\Big\}\cup\Big\{z\ \text{point in}\ \bigcup_{n\geqslant 0}P^{-n}(\alpha)\cap\overline{V}\Big\}$ where $V=I_{0,1}\cup I_{1,2}\cup I_{2,0}\cup I_{2,3}\cup I_{3,0}\cup I_{3,1}$.
\end{itemize}
Consequently, we can show as above that the itinerary map $h_{2}:\{z\in V\,/\,\forall n\geqslant 0,\,P^{n}(z)\in V\}\rightarrow\Sigma^{\star}$ extends to a homeomorphism from $\JJ(\HH_{P})$ to $\Sigma^{\star}$ which conjugates the dynamics of $P$ and $\sigma$.

Finally, taking $h=h_{2}^{-1}\circ h_{1}$ concludes the proof. 
\end{proof}

Remark that the proof of Theorem \ref{ThmMain1} is almost completed. Indeed point \textbf{(i)} comes from Lemma \ref{LemDegree} while points \textbf{(ii)} and \textbf{(iii)} follows from Lemma \ref{LemEncoding} (since $\mathcal{A}$ is, by definition, the set of continua $J$ in $\mathcal{A^{\star}}$ such that $J$ is not eventually mapped under iteration to the fixed continuum $J_{\alpha}$, or equivalently, such that $h(J)$ is not eventually mapped under iteration to the fixed branching point $\alpha$). It only remains to prove that $\mathcal{A^{\star}}$ is actually the set $\JJcrit(f)$ of all postcritically separating Julia components of $f$.

\begin{lem}\label{LemCriticallySeparating}
The following equality of sets holds.
$$\mathcal{A^{\star}}=\JJcrit(f)$$
\end{lem}

\begin{proof}
Recall that the postcritical set is contained in the forward invariant set $E=D(\beta_{1,2})\cup D(\beta_{2,3}^{-})\cup D(\delta_{3,c}^{-})\cup A(\gamma_{0,1},\gamma_{3,1})$ (see Lemma \ref{LemQcExtension} and Figure \ref{FigFinalDynamics}) and each point of the super-attracting cycle $\{z_{0},z_{1},z_{2},z_{3}\}$ lies in a different connected component of $E$. In particular $J(f)$ is the set of all points whose orbit remains in $\CC-E=\overline{A_{0}}\cup\overline{A_{1}}\cup\overline{A_{2}}\cup\overline{A_{3}}\cup K_{\alpha}$ where $K_{\alpha}$ is the complement in $\CC$ of $B(\widehat{z_{0}})\cup B(\widehat{z_{1}})\cup B(\widehat{z_{2}})$ (see Figure \ref{FigEncoding}).

It follows that every element $J$ in $\mathcal{A}$ is a Julia component. Moreover $J$ is postcritically separating as limit set of nested essential subannuli which separate each the super-attracting cycle $\{z_{0},z_{1},z_{2},z_{3}\}$ (see proof of Lemma \ref{LemEncoding}). Therefore $\mathcal{A}\subset\JJcrit(f)$.

Similarly, every element $J$ in $\mathcal{A}_{\alpha}$ is a Julia component. Moreover recall that $J$ intersects $\overline{U}$ along a limit set of nested essential subannuli which separate each the super-attracting cycle $\{z_{0},z_{1},z_{2},z_{3}\}$ (see proof of Lemma \ref{LemEncoding}). Therefore $\mathcal{A}_{\alpha}\subset\JJcrit(f)$ and $\mathcal{A}^{\star}=\mathcal{A}\cup\mathcal{A}_{\alpha}\subset\JJcrit(f)$.

Conversely, let $J$ be a postcritically separating Julia component of $f$. Remark that $J$ is not contained in $K_{\alpha}-J_{\alpha}$. Indeed, recall that every connected component of $\CC-J_{\alpha}$ is simply connected (see Lemma \ref{LemBranchingMap}) and that $\partial K_{\alpha}=\alpha_{0}\cup\alpha_{1}\cup\alpha_{2}\subset J_{\alpha}$, therefore every connected compact subset of any connected component of $K_{\alpha}-J_{\alpha}$ does not separate the postcritical points. Consequently either $J$ is $J_{\alpha}\in\mathcal{A}_{\alpha}\subset\mathcal{A}^{\star}$ or $f^{n}(J)$ stays in $\overline{A_{0}}\cup\overline{A_{1}}\cup\overline{A_{2}}\cup\overline{A_{3}}$ for every $n\geqslant 0$. Assume that $J$ is not $J_{\alpha}$.

Recall that every connected component of the preimage under $f$ of $\overline{A_{0}}\cup\overline{A_{1}}\cup\overline{A_{2}}\cup\overline{A_{3}}$ which is contained in this compact union, is contained either in $\overline{U}$ or in some connected components of $\overline{f^{-1}(A_{3})}$ included in $A_{3}$ (from Lemma \ref{LemLastStep1}, see Figure \ref{FigLastStep}), says $\overline{A'_{3,3}}$. However every $A'_{3,3}$ is not contained in $A_{3}$ as essential subannulus, and hence does not separate the postcritical points. In particular $J$ is not contained in any $\overline{A'_{3,3}}$. Furthermore, $J$ can not eventually fall in some $\overline{A'_{3,3}}$ after some iterations of $f$, otherwise $f^{n}(J)$ would not be postcritically separating for some $n\geqslant 0$ contradicting the fact that $J$ is postcritically separating. It follows that $f^{n}(J)$ stays in $\overline{U}$ for every $n\geqslant 0$ and hence $J\in\mathcal{A}^{\star}$ that concludes the proof.
\end{proof}

\subsection{Topology of buried Julia components}\label{SubSecTopology}

Existence of each of the three types of buried Julia components which occurs in $J(f)$ is shown in this section, that proofs Theorem \ref{ThmMain2}.

\begin{lem}[Point type buried Julia components]
There exist uncountably many buried Julia components in $J(f)$ which are points.
\end{lem}

\begin{proof}
Let $A'_{3,3}=A(\beta_{3,3}^{+},\beta_{3,3}^{-})$ be a connected component of $f^{-1}(A_{3})$ contained in $A_{3}=A(\beta_{3}^{+},\beta_{3}^{-})$ (from Lemma \ref{LemLastStep1}, see Figure \ref{FigLastStep}) where $\beta_{3,3}^{+}$ and $\beta_{3,3}^{-}$ are preimages of $\beta_{3}^{+}$ and $\beta_{3}^{-}$, respectively. Recall that $A'_{3,3}$ is not contained in $A_{3}$ as essential subannulus. In particular, the connected component of $\CC-\beta_{3,3}^{+}$ containing $A'_{3,3}$ is an open disk $D(\beta_{3,3}^{+})$ contained in $A_{3}$ and such that $f|_{D(\beta_{3,3})}:D(\beta_{3,3}^{+})\rightarrow D$ is a homeomorphism where $D=D(\beta_{3}^{+})$ is the open disk bounded by $\beta_{3}^{+}$ and containing $A_{3}$.

Using notations coming from the proof of Lemma \ref{LemEncoding}, consider the subannulus $A_{3,0,1,2,3}$ contained in $A_{3}$ as essential subannulus and such that $f^{4}|_{A_{3,0,1,2,3}}:A_{3,0,1,2,3}\rightarrow A_{3}$ is a degree $d_{3}d_{0}d_{1}d_{2}$ covering. Since assumption (\ref{H2}) implies that at least two of weights $d_{0}$, $d_{1}$, $d_{2}$, and $d_{3}$ are $\geqslant 2$ (see Definition \ref{DefTransitionMatrix}), it follows that this degree is $\geqslant 2$ and hence, there are at least 2 disjoint preimages under $f^{4}|_{A_{3,0,1,2,3}}$ of $D(\beta_{3,3}^{+})$ in $A_{3,0,1,2,3}\subset A_{3}\subset D$, says $D_{0}$ and $D_{1}$.

Finally we have two disjoint open disks $D_{0}$ and $D_{1}$ in $D$ such that $f^{5}|_{D_{0}}:D_{0}\rightarrow D$ and $f^{5}|_{D_{1}}:D_{1}\rightarrow D$ are homeomorphisms. It is then a classical exercise to prove that the non-escaping set
$$\mathcal{D}=\{z\in D_{0}\cup D_{1}\,/\,\forall n\geqslant 0, (f^{5})^{n}\in D_{0}\cup D_{1}\}$$
is a Cantor set homeomorphic to the space of all sequences of two digits $\Sigma_{2}=\{0,1\}^{\N}$. In particular, $\mathcal{D}$ contains uncountably many points. Furthermore every point in $\mathcal{D}$ is a buried point in $J(f)$ since $A_{3}\subset D$ contains infinitely many postcritically separating Julia components.
\end{proof}

\begin{lem}[Circle type buried Julia components]
There exist uncountably many buried Julia components in $J(f)$ which are wandering Jordan curves.
\end{lem}

\begin{proof}
This is mostly a consequence of the main result in \cite{RationalMapsDisconnectedJuliaSet} claiming that every wandering Julia component of a geometrically finite rational map is either a point or a Jordan curve. Here our map $f$ is hyperbolic (from Lemma \ref{LemDegree}) therefore every wandering Julia component in $\JJcrit(f)$ must be a Jordan curve (since a point is obviously not postcritically separating). Moreover, according to the proof of Lemma \ref{LemEncoding}, the set of wandering Julia components in $\JJcrit(f)$ exactly corresponds to the set of all the infinite sequences in $\Sigma^{\star}$ which are not eventually periodic. In particular, there are uncountably many such Julia components. Finally, uncountably many of them must be buried since the Fatou set only has countably many Fatou domains and each of them only has countably many Jordan curves as connected components of its boundary.
\end{proof}

\begin{lem}[Complex type buried Julia components]
$J_{\alpha}$ and all its countably many preimages, are buried Julia components in $J(f)$.
\end{lem}

\begin{proof}
Coming back to the proof of Lemma \ref{LemEncoding}, remark that every infinite sequence in $S_{\alpha}$ is not isolated in $\Sigma$. Therefore, $\alpha_{k}$ has no intersection with the boundary of any Fatou domain contained in $B(\widehat{z_{k}})$ for every $k\in\{0,1,2\}$. It remains to show that $J_{\alpha}$ has no intersection with the boundary of any Fatou domain in $K_{\alpha}=\CC-(B(\widehat{z_{0}})\cup B(\widehat{z_{1}})\cup B(\widehat{z_{2}}))$. Recall that every connected component of $K_{\alpha}-J_{\alpha}$, that is a connected component of $\CC-J_{\alpha}$, is eventually mapped under iteration onto $B(\widehat{z_{k}})$ for some $k\in\{0,1,2\}$ (since $f$ is defined to be $\widehat{f}$ on $K_{\alpha}\subset D(\beta_{0,1})$). By continuity of $f$, it follows that $J_{\alpha}$ has no intersection with the boundary of any Fatou domain contained in any connected component of $K_{\alpha}-J_{\alpha}$. Consequently $J_{\alpha}$ is buried. The same holds as well for every preimage of $J_{\alpha}$ by continuity of $f$.
\end{proof}

%*******************************************************************************************

\section{Explicit formula in the cubic case}\label{SecFormula}

In this section, we proof Theorem \ref{ThmPersianCarpet} stated in the introduction (see Section \ref{SecIntroduction}). Firstly we show that a particular choice of the weight function $w$ gives a rational map of degree $3$ (in Lemma \ref{LemExampleDegree3}). Then we compute an explicit formula for this particular example.

\begin{lem}\label{LemExampleDegree3}
The following weight function on the set of edges of $\HH_{P}$
$$(d_{0},d_{1},d_{2},d_{3})=(1,2,2,1)$$
satisfies assumptions (\ref{H1}) and (\ref{H2}) from Theorem \ref{ThmMain1} and Theorem \ref{ThmMain2}. In particular there are some rational maps of degree $3$ whose Julia set contains buried Julia components of several types: \\
\begin{tabularx}{\textwidth}{rX}
	\emph{(point type)} & uncountably many points; \\
	\emph{(circle type)} & uncountably many Jordan curves; \\
	\emph{(complex type)} & countably many preimages of a fixed Julia component which is quasiconformally homeomorphic to the connected Julia set of $\widehat{f}:z\mapsto\frac{1}{(z-1)^{2}}$.
\end{tabularx}
\end{lem}

\begin{proof}
Assumption (\ref{H1}) is obviously satisfied, indeed
$$\widehat{d}=\frac{1}{2}(d_{0}+d_{1}+d_{2}-1)=\frac{1}{2}(1+2+2-1)=2=\max\{d_{0},d_{1},d_{2}\}.$$

For assumption (\ref{H2}), the transition matrix (see Definition \ref{DefTransitionMatrix}) for this choice of weight function is given by
$$M=\left(\begin{array}{cccc} 0 & 1 & 0 & 0 \\ 0 & 0 & \frac{1}{2} & 0 \\ \frac{1}{2} & 0 & 0 & \frac{1}{2} \\ 1 & 1 & 0 & 0 \end{array}\right)$$
and an easy computation shows that $\lambda(\HH_{P},w)$ is the largest root of $X^{4}-\frac{1}{2}X-\frac{1}{4}$ that is $\lambda(\HH_{P},w)\approx 0.918<1$.

Applying Theorem \ref{ThmMain1} and Theorem \ref{ThmMain2} gives a rational map of degree $\widehat{d}+d_{3}=2+1=3$.

Furthermore, recall that the rational map $\widehat{f}$ which appears in Theorem \ref{ThmMain2} has degree $\widehat{d}=2$ and has only one critical orbit which is a super-attracting cycle $\{\widehat{z_{0}},\widehat{z_{1}},\widehat{z_{2}}\}$ of period $3$ such that the local degrees of $\widehat{f}$ at $\widehat{z_{0}}$, $\widehat{z_{1}}$ and $\widehat{z_{2}}$ are $d_{0}=1$, $d_{1}=2$ and $d_{2}=2$, respectively. Up to conjugation by a Möbius map, we may assume that $\widehat{z_{0}}=0$, $\widehat{z_{1}}=1$ and $\widehat{z_{2}}=\infty$. It turns out that there is then only one such quadratic rational map which is $\widehat{f}:z\mapsto\frac{1}{(z-1)^{2}}$.
$$\xymatrix{\widehat{z_{0}}=0 \ar[rr]^{1:1} && \widehat{z_{1}}=1 \ar[rr]^{2:1} && \widehat{z_{2}}=\infty \ar@/_{2pc}/[llll]_{2:1}}$$
\end{proof}

Remark that this choice of weight function is the only one which gives a degree 3 and which satisfies assumptions (\ref{H1}) and (\ref{H2}).

The construction by quasiconformal surgery detailed in Section \ref{SecConstruction} does not provide an algebraic formula for the rational map $f$ in Theorem \ref{ThmMain1} and Theorem \ref{ThmMain2}. Furthermore the degree $\widehat{d}+d_{3}$ of $f$ increases quickly with the weight function $w$ so the algebraic relations behind are complicated to study. However the particular rational map of degree 3 coming from Lemma \ref{LemExampleDegree3} is simple enough to allow a computation by hand of an algebraic formula.

Let $f$ be a rational map coming from the construction detailed in Section \ref{SecConstruction} for the particular choice of weight function in Lemma \ref{LemExampleDegree3} . Recall that the local degrees of $f$ at $z_{1}$, $z_{2}$ and $z_{3}$ are $d_{1}=2$, $d_{2}=2$ and $d_{3}=1$, respectively. In particular, $z_{1}$ and $z_{2}$ are simple critical points. It remains $d_{0}+d_{3}=1+1=2$ critical points counted with multiplicity coming from definition of $f$ near $z_{0}$ (see Lemma \ref{LemFolding}), namely two simple critical points, one is $z_{0}$ by construction and the orbit of the other one accumulates the super-attracting cycle $\{z_{0},z_{1},z_{2},z_{3}\}$.

Up to conjugation by a Möbius map, we assume that $z_{1}=1$, $z_{2}=\infty$ and $z_{3}=0$. So 1 and $\infty$ are critical points whereas 0 is a singular point. In order to simplify notations, denote by $\lambda$ the critical point $z_{0}$ ($\lambda$ will be the parameter of our family) and by $\lambda'$ the last critical point.
$$\xymatrix @!0 @R=1.5pc @C=4pc {
	z_{0}=\lambda \ar[rr]^{2:1} && z_{1}=1 \ar[rr]^{2:1} && z_{2}=\infty \ar[rr]^{2:1} && z_{3}=0 \ar@/_{3pc}/[llllll]_{1:1} \\
	\hspace{1cm}\lambda' \ar[rr]^{2:1} && \dots &&&&
}$$ % ATTENTION : bidouille pour aligner lambda et lambda'

Since $f$ has degree 3, it is of the form
$$f:z\mapsto\dfrac{a_{3}z^{3}+a_{2}z^{2}+a_{1}z+a_{0}}{b_{3}z^{3}+b_{2}z^{2}+b_{1}z+b_{0}}.$$
Since $z_{1}=1$ is mapped to $z_{2}=\infty$ with a local degree 2, the denominator may factor as
$$f:z\mapsto\dfrac{a_{3}z^{3}+a_{2}z^{2}+a_{1}z+a_{0}}{(z-1)^{2}(b'_{1}z+b'_{0})}.$$
We do likewise for $z_{2}=\infty$ which is mapped to $z_{3}=0$ with a local degree 2.
$$f:z\mapsto\dfrac{a_{1}z+a_{0}}{(z-1)^{2}(b'_{1}z+b'_{0})}.$$
Now use the fact that $z_{3}=0$ is mapped to $z_{0}=\lambda$ to get
\begin{equation}\label{EqFormula}
	f:z\mapsto\dfrac{a_{1}z+\lambda}{(z-1)^{2}(b'_{1}z+1)}.
\end{equation}
It remains two informations coming from the fact that $z_{0}=\lambda$ is mapped to $z_{1}=1$ with a local degree 2. Namely $f(\lambda)=1$ and $f'(\lambda)=0$ which lead to the two following equations satisfied by $a_{1}$ and $b'_{1}$.
$$\left\{\begin{array}{rcl} (\lambda-1)^{2}(\lambda b'_{1}+1) & = & \lambda(a_{1}+1) \\ a_{1}(\lambda-1)^{2}(\lambda b'_{1}+1) & = & \lambda (a_{1}+1)\Big[(3\lambda^{2}-4\lambda+1)b'_{1}+2(\lambda-1)\Big] \end{array}\right.$$
Remark that we may easily simplify the second equation by using the first one (luckily)
$$\left\{\begin{array}{rcl} (\lambda-1)^{2}(\lambda b'_{1}+1) & = & \lambda(a_{1}+1) \\ a_{1} & = & (3\lambda^{2}-4\lambda+1)b'_{1}+2(\lambda-1) \end{array}\right.$$
or equivalently
$$\left\{\begin{array}{rcrcl}
	\lambda a_{1} & - & \lambda(1-\lambda)^{2}b'_{1} & = & 1-3\lambda+\lambda^{2} \\
	a_{1} & - & (1-\lambda)(1-3\lambda)b'_{1} & = & -2+2\lambda
\end{array}\right.$$
and solving this linear system of two equations gives
$$\left\{\begin{array}{rcl}
	a_{1} & = & \dfrac{(1-3\lambda)(1-3\lambda+\lambda^{2})-\lambda(1-\lambda)(-2+2\lambda)}{\lambda(1-3\lambda)-\lambda(1-\lambda)}=\dfrac{1-4\lambda+6\lambda^{2}-\lambda^{3}}{-2\lambda^{2}} \\
	b'_{1} & = & \dfrac{(1-3\lambda+\lambda^{2})-\lambda(-2+2\lambda)}{-\lambda(1-\lambda)^{2}+\lambda(1-\lambda)(1-3\lambda)}=\dfrac{1-\lambda-\lambda^{2}}{-2\lambda^{2}(1-\lambda)}
\end{array}\right..$$
Finally, putting these expressions in expression (\ref{EqFormula}) leads to the following formula for $f$ which depends on the parameter $\lambda$.
$$\boxed{f_{\lambda}:z\mapsto\dfrac{(1-\lambda)\Big[(1-4\lambda+6\lambda^{2}-\lambda^{3})z-2\lambda^{3}\Big]}{(z-1)^{2}\Big[(1-\lambda-\lambda^{2})z-2\lambda^{2}(1-\lambda)\Big]}}$$

Remark that $f_{\lambda}(z)=\frac{1}{(z-1)^{2}}(1-4\lambda+O_{\lambda\rightarrow 0}(\lambda^{2}))$ for every complex number $z$, thus $f_{\lambda}$ is actually a particular perturbation of $f_{0}=\widehat{f}:z\mapsto\frac{1}{(z-1)^{2}}$.

Some more computations provide an algebraic formula for the critical point $\lambda'$, namely
$$\lambda'=-\dfrac{\lambda(1-6\lambda+11\lambda^{2}-10\lambda^{3}+5\lambda^{4})}{(1-\lambda-\lambda^{2})(1-4\lambda+6\lambda^{2}-\lambda^{3})}=-\lambda+\underset{\lambda\rightarrow 0}{O}(\lambda^{2}).$$

According to the construction detailed in Section \ref{SecConstruction}, there exist some choices of $\lambda$ such that $f_{\lambda}$ satisfies Theorem \ref{ThmPersianCarpet}. Recall that the two critical points $z_{0}=\lambda$ and $\lambda'\sim_{\lambda\rightarrow 0}-\lambda$ should lie in $B(\widehat{z_{0}})$ (see Section \ref{SecConstruction}), and hence near $\widehat{z_{0}}$ which corresponds to $z_{3}=0$. Indeed, we can roughly prove for every $|\lambda|>0$ small enough that
\begin{itemize}
	\item $f_{\lambda}(\lambda')$ lies in a disk centered at $z_{1}=1$ and of radius of order $|\lambda|$;
	\item the image under $f_{\lambda}$ of a disk centered at $z_{1}=1$ and of radius of order $|\lambda|$ is contained in the complement of a disk centered at $0$ (thus containing $z_{2}=\infty$) and of radius of order $|\lambda|^{-2}$;
	\item the image under $f_{\lambda}$ of the complement of a disk centered at $0$ (thus containing $z_{2}=\infty$) and of radius of order $|\lambda|^{-2}$ is contained in a disk centered at $z_{3}=0$ and of radius of order $|\lambda|^{4}$;
	\item the image under $f_{\lambda}$ of a disk centered at $z_{3}=0$ and of radius of order $|\lambda|^{4}$ is contained in a disk centered at $z_{0}=\lambda$ and of radius of order $|\lambda|^{2}$;
	\item the image under $f_{\lambda}$ of a disk centered at $z_{0}=\lambda$ and of radius of order $|\lambda|^{2}$ is contained in a disk centered at $z_{1}=1$ and of radius of order $|\lambda|^{3}$.
\end{itemize}

It turns out that the orbit of the critical point $\lambda'$ accumulates the super-attracting cycle $\{z_{0},z_{1},z_{2},z_{3}\}$ for every $|\lambda|>0$ small enough. Consequently, we may encode the exchanging dynamics of Julia components of $f_{\lambda}$ as it is explained in Section \ref{SecProperties}, proving that $f_{\lambda}$ satisfies Theorem \ref{ThmPersianCarpet} for every $|\lambda|>0$ small enough.

Numerically, picking any parameter $\lambda$ in the  big hyperbolic component surrounding $0$ of the parameter space of the family $f_{\lambda}$ (see Figure \ref{FigParameterPersianCarpet}.a) provides a Persian Carpet example in the dynamical plane (see Figure \ref{FigParameterPersianCarpet}.b).

\begin{figure}[!htb]
	\begin{center}
	\includegraphics[width=\textwidth]{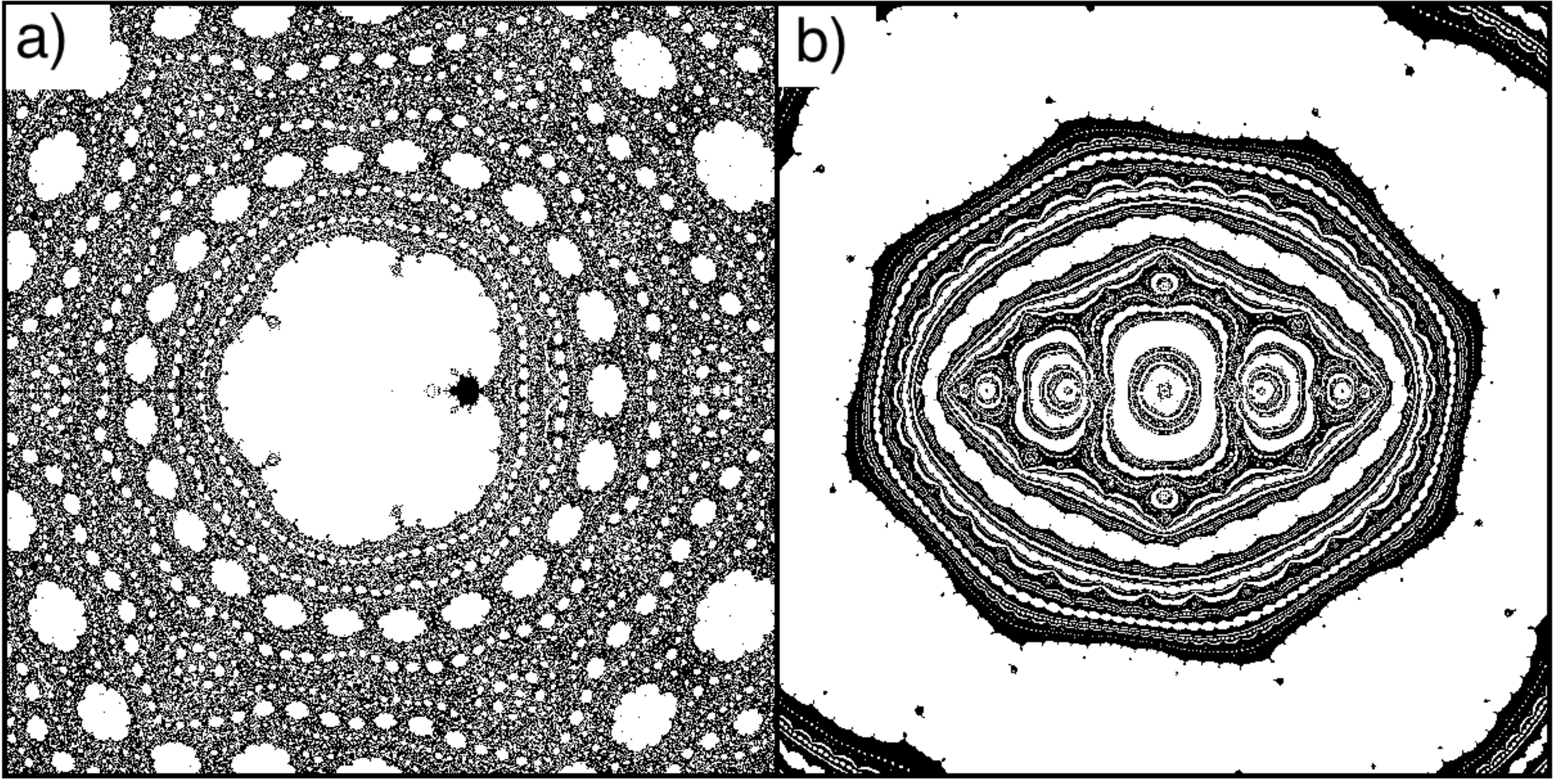}
	\caption{\textbf{a)} The parameter plane of $f_{\lambda}$ for $|\lambda|\lessapprox 10^{-2}$, that includes the bifurcation locus (in black) and hyperbolic parameters (in white). $0$ is at the center of the picture and the big hyperbolic component around corresponds to the Persian carpets. \\ \textbf{b)} The dynamical plane for $\lambda\approx10^{-3}$, that includes the Persian carpet $J(f_{\lambda})$ (in black) and the Fatou set (in white).}
	\label{FigParameterPersianCarpet}
	\end{center}
\end{figure}

%*******************************************************************************************

\section{Appendix}\label{SecAppendix}

In this section, we collect some technical results used in the construction of Section \ref{SecConstruction}.

\subsection{A particular solution of the Hurwitz problem}

The first result of this section deals with the Hurwitz problem on the topological sphere $\SS$. Namely given an abstract branch data of degree $d\geqslant 2$, that is a table of positive integers $\mathcal{D}=(d_{i,j})_{(i,j)\in\mathcal{I}}$ where $\mathcal{I}=\{(i,j)\ /\ i\in\{1,2,\dots,n\}\ \text{and}\ j\in\{1,2,\dots,k_{i}\}\}$ for some positive integers $n$, $k_{1}$, $k_{2}$, \dots, $k_{n}$ and such that for every $i\in\{1,2,\dots,n\}$:
\begin{equation}\label{HurwitzCondition}
	d_{i,j}\geqslant 2\ \text{for some}\ j\in\{1,2,\dots,k_{i}\},\ \text{and}\ \sum_{j=1}^{k_{i}}d_{i,j}=d,
\end{equation}
we consider the question on realizability of this abstract branch data by a branched covering on $\SS$, that is the existence of a degree $d$ branched covering $H:\SS\rightarrow\SS$ and a finite collection of distinct points $X=\{x_{i,j}\,/\,(i,j)\in\mathcal{I}\}$ in $\SS$ such that:
\begin{itemize}
	\item $\forall(i,j)\in\mathcal{I}$, $H(x_{i,j})=y_{i}$ for some $y_{i}\in\SS$;
	\item $H|_{\SS-X}:\SS-X\rightarrow\SS-\{y_{i}\,/\,i\in\{1,2,\dots,n\}\}$ is a degree $d$ covering;
	\item $\forall(i,j)\in\mathcal{I}$, the local degree of $H$ at $x_{i,j}$ is $d_{i,j}$.
\end{itemize}

Adolf Hurwitz has proved (see \cite{HurwitzProblem}) that the solution is as follows. Let $S_{d}$ be the symmetric group of all permutations of $\{1,2,\dots,d\}$. Then $\mathcal{D}$ is realizable if and only if there exist permutations $\sigma_{1}$, $\sigma_{2}$, \dots, $\sigma_{n}$ in $S_{d}$ such that:
\begin{description}
	\item[(i)] $\forall i\in\{1,2,\dots,n\}$, $\sigma_{i}$ is a product of $k_{i}$ disjoint cycles of length $d_{i,1}$, $d_{i,2}$, \dots, $d_{i,k_{i}}$;
	\item[(ii)] $\sigma_{1}\sigma_{2}\dots\sigma_{n}=1$ in $S_{d}$;
	\item[(iii)] $\langle\sigma_{1},\sigma_{2},\dots,\sigma_{n}\rangle$ transitively acts on $\{1,2,\dots,d\}$.
\end{description}

However, such algebraic conditions are rather difficult to verify for an arbitrary abstract branch data $\mathcal{D}$. Easier sufficient and necessary conditions have been provided in some specific cases (see for instance \cite{RealizabilityBranchedCoverings} or \cite{RealizabilityBaranski}). The following lemma gives the solution in a very special case involved in Lemma \ref{LemBranchingMap}.

\begin{lem}[An Hurwitz solution]\label{LemHurwitz}
Let $\mathcal{D}$ be an abstract branch data of degree $d\geqslant 2$ such that $n=3$ and $d_{i,j}=1$ for every $i\in\{1,2,3\}$ and $j\geqslant 2$. Then $\mathcal{D}$ is realizable if and only if the following condition is satisfied.
\begin{equation}\label{H1'}
	d=\frac{1}{2}(d_{1,1}+d_{2,1}+d_{3,1}-1) \tag{H1'}
\end{equation}
\end{lem}

Remark that in this special case, the abstract branch data $\mathcal{D}$ is uniquely determined by a degree $d\geqslant 2$ together with three positive integers $d_{1,1}$, $d_{2,1}$, and $d_{3,1}$ such that $2\leqslant d_{i,1}\leqslant d$ for every $i\in\{1,2,3\}$.

\begin{proof}
The necessity of condition (\ref{H1'}) comes from the Riemann-Hurwitz formula. Indeed if there exists a degree $d$ branched covering $h:\SS\rightarrow\SS$ which realizes $\mathcal{D}$ then the number of critical points counted with multiplicity is given by
$$2d-2=(d_{1,1}-1)+(d_{2,1}-1)+(d_{3,1}-1)$$
since there are no more critical points than $x_{1,1}$, $x_{2,1}$, and $x_{3,1}$ from assumption.

For the sufficiency of condition (\ref{H1'}), consider the two following cycles in $S_{d}$
$$\sigma_{1}=(1,2,\dots,d_{1,1})\ \text{and}\ \sigma_{2}=(d,d-1,\dots,d-d_{2,1}+1)$$
of length $d_{1,1}$ and $d_{2,1}$ respectively. Notice that $d_{1,1}=(d-d_{2,1}+1)+(d-d_{3,1})$ from condition (\ref{H1'}), and $d_{3,1}\leqslant d$ from assumption (\ref{HurwitzCondition}). Therefore $d-d_{2,1}+1\leqslant d_{1,1}$ and in particular $\langle\sigma_{1},\sigma_{2}\rangle$ transitively acts on $\{1,2,\dots,d\}$.

A simple computation shows that the composition of $\sigma_{1}$ and $\sigma_{2}$ (with $\sigma_{1}$ performed first) is given by
$$\sigma_{1}\sigma_{2}=(1,2,\dots,d-d_{2,1},d,d-1,\dots,d_{1,1})$$
which is a cycle of length $(d-d_{2,1})+(d-d_{1,1}+1)=d_{3,1}$ from condition (\ref{H1'}). Denote by $\sigma_{3}$ the inverse permutation $(\sigma_{1}\sigma_{2})^{-1}$ which is a cycle of length $d_{3,1}$ as well.

Finally we get three permutations $\sigma_{1}$, $\sigma_{2}$, and $\sigma_{3}$ in $S_{d}$ which satisfy the Hurwitz conditions. Indeed \textbf{(i)} holds since $\sigma_{i}$ is a cycle of length $d_{i,1}$ for every $i\in\{1,2,3\}$, \textbf{(ii)} directly follows from definition of $\sigma_{3}$, and $\textbf{(iii)}$ holds since $\langle\sigma_{1},\sigma_{2}\rangle=\langle\sigma_{1},\sigma_{2},\sigma_{3}\rangle$ transitively acts on $\{1,2,\dots,d\}$.
\end{proof}

\subsection{An inverse Grötzsch's inequality}

The following useful result is due to Cui Guizhen and Tan Lei \cite{ThurstonSubHyperbolic}. It is the key ingredient of the proof of Lemma \ref{LemEquip}.

\begin{lem}[Inverse Grötzsch's inequality]\label{LemGrotzch}
Let $D,D'$ be two disjoint marked hyperbolic disks in $\CC$ whose boundaries (not necessarily disjoint) are respectively denoted by $\alpha,\alpha'$. Then there exists a positive constant $C>0$ such that for every pair of equipotentials $\beta$ in $D$ and $\beta'$ in $D'$ the following inequalities hold:
$$\modulus(A(\alpha,\beta))+\modulus(A(\alpha',\beta'))\leqslant\modulus(A(\beta,\beta'))\leqslant\modulus(A(\alpha,\beta))+\modulus(A(\alpha',\beta'))+C.$$
\end{lem}

The left hand side is the classical Grötzsch's inequality. The right hand side is a consequence of Koebe $1/4$ Theorem. We refer readers to \cite{ThurstonSubHyperbolic} for a complete proof.

\subsection{An annulus-disk holomorphic map}

The following lemma is a technical ingredient in the construction of Section \ref{SecConstruction} needed to holomorphically map an annulus onto a disk (see Lemma \ref{LemFolding}). It is very similar to the key lemma in \cite{DiskAnnulusSurgery} (see also \cite{QuasiconformalSurgeryBook}) about an annulus-disk branched covering. However, our annulus-disk map here requires to be holomorphic (see Lemma \ref{LemQcExtension} and Lemma \ref{LemQuasiconformalSurgery}).

\begin{lem}[Annulus-disk holomorphic map]\label{LemAnnulusDisk}
Let $n,n'$ be two positive integers. Then there exists a holomorphic branched covering $G:A(\gamma,\gamma')\rightarrow\D$ from an open annulus in $\CC$ bounded by a pair of disjoint quasicircles $\gamma,\gamma'$ onto the open unit disk $\D$ such that:
\begin{description}
	\item[(i)] $G$ has degree $n+n'$ and has $n+n'$ critical points counted with multiplicity;
	\item[(ii)] $G$ continuously extends to $\gamma\cup\gamma'$ by a degree $n$ covering $G|_{\gamma}:\gamma\rightarrow\partial\D$ and a degree $n'$ covering $G|_{\gamma'}:\gamma'\rightarrow\partial\D$;
	\item[(iii)] $\modulus(A(\gamma,\gamma'))\leqslant 1$.
\end{description}
\end{lem}

\begin{proof}
There are many ways to prove the existence of such a map. Here this proof uses the properties of the McMullen's family
$$g_{0,\lambda}:z\mapsto z^{n}+\frac{\lambda}{z^{n'}}$$
for $|\lambda|>0$ small enough (see \cite{AutomorphismsRationalMaps} and \cite{SingularPerturbations} for a complete study of this family). Recall that $g_{0,\lambda}$ has degree $n+n'$, and its critical set contains $n+n'$ simple critical points of the form
$$c_{k}=\left(\frac{n'}{n}\right)^{1/(n+n')}|\lambda|^{1/(n+n')}e^{2ki\pi/(n+n')}\quad\text{where}\ k\in\{1,2,\dots,n+n'\}$$
(the other critical points are $\infty$ of  multiplicity $n-1$ if $n>1$ and $0$ of multiplicity $n'-1$ if $n'>1$). Moreover, the preimages of $0$ are of the form
$$g_{0,\lambda}^{-1}(0)=\{|\lambda|^{1/(n+n')}e^{2ki\pi/(n+n')}\,/\,k\in\{1,2,\dots,n+n'\}\}.$$

Let $A$ be the preimage of the open unit disk $\D$, namely $A=g_{0,\lambda}^{-1}(\D)$. We are going to prove that for every $|\lambda|>0$ small enough $A$ is connected and actually an open annulus separating $0$ and $\infty$. Indeed remark that for every $z\in\C$ with modulus $|z|=|\lambda|^{1/(n+n')}$ we have
$$|g_{0,\lambda}(z)|=\left|z^{n}+\frac{\lambda}{z^{n'}}\right|\leqslant|\lambda|^{n/(n+n')}+\frac{|\lambda|}{|\lambda|^{n'/(n+n')}}=2|\lambda|^{n/(n+n')}.$$
Similarly for every $k\in\{1,2,\dots,n+n'\}$ we have
$$|g_{0,\lambda}(c_{k})|\leqslant\left(\frac{n'}{n}\right)^{n/(n+n')}|\lambda|^{n/(n+n')}+\left(\frac{n}{n'}\right)^{n'/(n+n')}\frac{|\lambda|}{|\lambda|^{n'/(n+n')}}=C|\lambda|^{n/(n+n')}$$
with
$$C=\left(\frac{n'}{n}\right)^{n/(n+n')}+\left(\frac{n}{n'}\right)^{n'/(n+n')}=\left(\frac{n+n'}{n}\right)^{n/(n+n')}\times\left(\frac{n+n'}{n'}\right)^{n'/(n+n')}\leqslant 2$$
(by using the arithmetic-geometric mean inequality $x^{n/(n+n')}\times y^{n'/(n+n')}\leqslant \frac{n}{n+n'}x+\frac{n'}{n+n'}y$).

So if $\lambda$ is such that $0<2|\lambda|^{n/(n+n')}<1$ then $A$ contains the circle centered at $0$ and of radius $|\lambda|^{1/(n+n')}$ where all the preimages of $0$ lie, together with $n+n'$ simple critical points of $g_{0,\lambda}$. In particular, $A$ is a connected set which separates $0$ and $\infty$ and it follows from the Riemann-Hurwitz formula applied to the degree $n+n'$ branched covering $g_{0,\lambda}|_{A}:A\rightarrow\D$ that $A$ is an open annulus.

Now let $\gamma$ be the outer boundary of $A$, namely the boundary of the connected component of $\CC-\overline{A}$ containing $\infty$, and $\gamma'$ be the inner boundary of $A$, namely the boundary of the connected component of $\CC-\overline{A}$ containing $0$. It turns out that $A=A(\gamma,\gamma')$ and $G=g_{0,\lambda}|_{A}:A(\gamma,\gamma')\rightarrow\D$ satisfy \textbf{(i)}. The point \textbf{(ii)} follows from the fact that $g_{0,\lambda}$ realizes a degree $n$ (respectively $n'$) branched covering on the the connected component of $\CC-\overline{A}$ containing $\infty$ (respectively $0$) with no critical points on the boundary. Moreover $\gamma$ and $\gamma'$ are quasicircles as preimages of the unit circle $\partial\D$ by conformal maps.

For the point \textbf{(iii)}, remark that for every $R\geqslant 1$ we have:
$$\begin{array}{rcl}
	|z|\leqslant\frac{1}{R^{1/n'}}|\lambda|^{1/(n+n')} & \Rightarrow & |g_{0,\lambda}(z)|\geqslant\frac{|\lambda|}{|z|^{n'}}-|z|^{n}\geqslant|\lambda|^{n/(n+n')}\big(R-\frac{1}{R^{n}}\big), \\
	\text{and}\quad|z|\geqslant R^{1/n}|\lambda|^{1/(n+n')} & \Rightarrow & |g_{0,\lambda}(z)|\geqslant|z|^{n}-\frac{|\lambda|}{|z|^{n'}}\geqslant|\lambda|^{n/(n+n')}\big(R-\frac{1}{R^{n'}}\big).
\end{array}$$
In particular if $R=2|\lambda|^{-n/(n+n')}$, then $\max\{\frac{1}{R^{n}},\frac{1}{R^{n'}}\}\leqslant\frac{1}{R}<\frac{1}{2}R$ (since $\lambda$ was chosen so that $0<2|\lambda|^{n/(n+n')}<1$ that implies $R>4$) and hence
$$|z|\leqslant\frac{1}{R^{1/n'}}|\lambda|^{1/(n+n')}\ \text{or}\ |z|\geqslant R^{1/n}|\lambda|^{1/(n+n')}\Rightarrow|g_{0,\lambda}(z)|>|\lambda|^{n/(n+n')}\frac{1}{2}R=1.$$
Consequently the preimage $A=A(\gamma,\gamma')$ of the unit disk is contained as essential subannulus in the round annulus $\{z\in\C\,/\,\frac{1}{R^{1/n'}}|\lambda|^{1/(n+n')}<|z|<R^{1/n}|\lambda|^{1/(n+n')}\}$ and the Grötzsch's inequality gives
$$\modulus(A(\gamma,\gamma'))\leqslant\frac{1}{2\pi}\log\left(\frac{R^{1/n}|\lambda|^{1/(n+n')}}{\frac{1}{R^{1/n'}}|\lambda|^{1/(n+n')}}\right)=\frac{1}{2\pi}\left(\frac{1}{n}+\frac{1}{n'}\right)\log(R).$$
In particular, if $\lambda$ is fixed so that $2|\lambda|^{n/(n+n')}=\frac{4}{e^{\pi}}<1$ then $R=2|\lambda|^{-n/(n+n')}=e^{\pi}$ and
$$\modulus(A(\gamma,\gamma'))\leqslant\frac{1}{2}\left(\frac{1}{n}+\frac{1}{n'}\right)\leqslant 1.$$
\end{proof}

In the proof above, $1$ is obviously not the optimal upper bound for $\modulus(A(\gamma,\gamma'))$. The author guesses that this modulus is arbitrary small when $\lambda$ is close to $0$. But one can prove that the modulus of the smallest round annulus containing $A(\gamma,\gamma')$ as essential subannulus is bounded by below by a positive constant which does not depend on $\lambda$. The same happens if the open unit disk $\D$ is replaced by any euclidean open disk centered at $0$ and containing the critical values. However we do not need a sharper estimation than \textbf{(iii)} in this paper (see Lemma \ref{LemEquip} and the proof of Lemma \ref{LemFolding}).

\subsection{A separating quasicircle}

The following lemma is used to define the quasicircle $\delta_{c}^{+}$ in the construction of Section \ref{SecConstruction} (see the proof of Lemma \ref{LemLastStep1}).

\begin{lem}[Separating quasicircle]\label{LemConfGeomEx}
Let $A(\gamma,\gamma')$ be an open annulus in $\CC$ bounded by a pair of disjoint quasicircles $\gamma,\gamma'$, and let $a$ be a point in $A(\gamma,\gamma')$. Then there exists a quasicircle $\delta$ in $A(\gamma,\gamma')$ which separates $a$ from $\gamma\cup\gamma'$ such that $\modulus(A(\gamma,\delta))$ is arbitrary small.
\end{lem}

The main idea is merely to define a quasicircle $\delta$ close enough to the boundary $\gamma$. We though provide an explicit proof which uses the definition of the modulus by extremal length (see \cite{ConformalInvariants}).

\begin{proof}
Up to a biholomorphic change of coordinates, we may assume that
$$\gamma=\{z\in\C\,/\,|z|=1\},\ \gamma'=\{z\in\C\,/\,|z|=e^{-2\pi\modulus(A(\gamma,\gamma'))}\},\ \text{and}\ a\in]e^{-2\pi\modulus(A(\gamma,\gamma'))},1[.$$
Fix $x$ to be the positive real number $(1+e^{-2\pi\modulus(A(\gamma,\gamma'))})/2$. For every $\varepsilon>0$ small enough, define $\delta_{\varepsilon}$ to be the euclidean circle centered at $x$ and of radius $1-x-\varepsilon$. Notice that $\delta_{\varepsilon}$ is included in $A(\gamma,\gamma')$ and that $\delta_{\varepsilon}$ separates $a$ from $\gamma\cup\gamma'$ for every $\varepsilon>0$ small enough.

For every angle $\theta$ (small enough), consider the path $\ell_{\theta}$ joining $\delta_{\varepsilon}$ to $\gamma$ of the form $\ell_{\theta}=\{z=re^{i\theta}\,/\,R_{\theta}\leqslant r\leqslant 1\}$ with $R_{\theta}>0$ maximal so that $R_{\theta}e^{i\theta}\in\delta_{\varepsilon}$. By classical results from euclidean geometry and trigonometry, we get
$$R_{\theta}=x\cos(\theta)+\sqrt{(1-x-\varepsilon)^{2}-x^{2}\sin^{2}(\theta)}.$$
Since $\theta\mapsto R_{\theta}$ is an even function with a local maximum at $\theta=0$, it follows for every $\varepsilon>0$ small enough that
\begin{eqnarray}
	\theta\in[-\sqrt{\varepsilon},\sqrt{\varepsilon}]\Longrightarrow R_{\theta}\geqslant R_{\sqrt{\varepsilon}} & = & x\cos(\sqrt{\varepsilon})+\sqrt{(1-x-\varepsilon)^{2}-x^{2}\sin^{2}(\sqrt{\varepsilon})} \nonumber\\
	& = & 1-\frac{2-x}{2(1-x)}\varepsilon+\underset{\varepsilon\rightarrow 0}{O}(\varepsilon^{2}) \nonumber\\
	& \geqslant & 1-C\varepsilon \label{IneqPath}
\end{eqnarray}
where $C$ is a positive constant fixed so that $C>\frac{2-x}{2(1-x)}$.

Now recall that the modulus of $A(\gamma,\delta_{\varepsilon})$ is given by the extremal length of the collection $L$ of rectifiable paths connecting $\delta_{\varepsilon}$ and $\gamma$, namely
$$\modulus(A(\gamma,\delta_{\varepsilon}))=\sup_{\rho}\frac{\left(\inf_{\ell\in L}\int_{\ell}\rho|dz|\right)^{2}}{\int_{A(\gamma,\delta_{\varepsilon})}\rho^{2}dxdy}$$
where the supremum is over all measurable functions $\rho:A(\gamma,\delta_{\varepsilon})\rightarrow[0,+\infty]$ such that $\int_{A(\gamma,\delta_{\varepsilon})}\rho^{2}dxdy<+\infty$. Let $\rho$ be such a measurable function. For every $\theta$ (small enough), we have
$$\inf_{\ell\in L}\int_{\ell}\rho|dz|\leqslant\int_{\ell_{\theta}}\rho|dz|=\int_{R_{\theta}}^{1}\rho(re^{i\theta})dr$$
that leads, integrating over $\theta\in[-\sqrt{\varepsilon},\sqrt{\varepsilon}]$ and applying the Cauchy-Schwarz inequality, to
\begin{eqnarray*}
	2\sqrt{\varepsilon}\inf_{\ell\in L}\int_{\ell}\rho|dz| & \leqslant & \left(\int_{-\sqrt{\varepsilon}}^{\sqrt{\varepsilon}}\int_{R_{\theta}}^{1}\rho(re^{i\theta})^{2}rdrd\theta\right)^{1/2}\left(\int_{-\sqrt{\varepsilon}}^{\sqrt{\varepsilon}}\int_{R_{\theta}}^{1}\frac{1}{r}drd\theta\right)^{1/2}\\
	& \leqslant & \left(\int_{A(\gamma,\delta_{\varepsilon})}\rho^{2}dxdy\right)^{1/2}\left(\int_{-\sqrt{\varepsilon}}^{\sqrt{\varepsilon}}\log\left(\frac{1}{R_{\theta}}\right)d\theta\right)^{1/2}.
\end{eqnarray*}
Therefore it follows from the inequality (\ref{IneqPath}) that
\begin{eqnarray*}
	\frac{\left(\inf_{\ell\in L}\int_{\ell}\rho|dz|\right)^{2}}{\int_{A(\gamma,\delta_{\varepsilon})}\rho^{2}dxdy} & \leqslant & \frac{1}{4\varepsilon}\int_{-\sqrt{\varepsilon}}^{\sqrt{\varepsilon}}\log\left(\frac{1}{R_{\theta}}\right)d\theta\\
	& \leqslant & \frac{1}{4\varepsilon}\int_{-\sqrt{\varepsilon}}^{\sqrt{\varepsilon}}\log\left(\frac{1}{1-C\varepsilon}\right)d\theta=\frac{1}{2\sqrt{\varepsilon}}\log\left(\frac{1}{1-C\varepsilon}\right).
\end{eqnarray*}
Finally we take the supremum over all measurable functions $\rho$ to get
$$\modulus(A(\gamma,\delta_{\varepsilon}))\leqslant\frac{1}{2\sqrt{\varepsilon}}\log\left(\frac{1}{1-C\varepsilon}\right)\underset{\varepsilon\rightarrow 0}{\sim}\frac{C}{2}\sqrt{\varepsilon}\underset{\varepsilon\rightarrow 0}{\rightarrow}0$$
that concludes the proof.
\end{proof}

%*******************************************************************************************

\bibliographystyle{alpha}
\bibliography{biblio}
\addcontentsline{toc}{section}{References}

\end{document}